\documentclass[10pt, journal]{IEEEtran}
\usepackage{cite, graphicx, subfigure,url,setspace,algorithm}
\usepackage{amsmath,amssymb}
\usepackage{bm}
\usepackage{subfigure}
\usepackage{CJK}

\newtheorem{assumption} {\cal A}
\newtheorem{proposition} {Proposition}
\newtheorem{remark}{Remark}

\newtheorem{corollary}{Corollary}
\begin{document}

\title{A Randomized Proper Orthogonal Decomposition Technique}

\author{
D. Yu,
\thanks{D. Yu is a Graduate Student Researcher, Department of Aerospace Engineering, Texas
  A\&M University, College Station}%
  \and
S. Chakravorty \thanks{S. Chakravorty is an Associate Professor of Aerospace Engineering, Texas A\&M University, College Station}}

\maketitle


\begin{abstract}
In this paper, we consider the problem of model reduction of large scale systems, such as those obtained through the discretization of PDEs. We propose a randomized proper orthogonal decomposition (RPOD) technique to obtain the reduced order models by randomly choosing a subset of the inputs/outputs of the system to construct a suitable small sized Hankel matrix from the full Hankel matrix. It is shown that the RPOD technique is computationally orders of magnitude cheaper when compared to techniques such as the Eigensystem Realization algorithm (ERA)/Balanced POD (BPOD) while obtaining the same information in terms of the number and accuracy of the dominant modes. The method is tested on several different advection-diffusion equations.
\end{abstract}

\section{INTRODUCTION}
In this paper, we consider the problem of model reduction of systems that are governed by partial differential equations (PDE). We propose a randomized version of the snapshot proper orthogonal decomposition (RPOD) technique that allows us to form an ROM of the PDE of interest in terms of the eigenfunctions of the PDE operator by randomly choosing a subset of the full Hankel matrix resulting from an input/output map of the PDE. The RPOD procedure requires orders of magnitude less computation when compared to the BPOD/ Eigensystem Realization Algorithm (ERA) procedure \cite{kung, juang} applied to the full-order Hankel matrix resulting from the discretization of a PDE with a large number of inputs and outputs.  The technique is applied to several different types of advection-diffusion equations to illustrate the procedure.\\

Model reduction has attracted considerable attention in the past several decades. It is a technique that constructs a lower-dimensional subspace to approximate the original higher-dimensional dynamic system. There are several well-known approaches to model reduction. The snapshot proper orthogonal decomposition (POD) technique, followed by a Galerkin projection has been used extensively in the Fluids community to produce reduced order models (ROMs) of fluid physics phenomenon such as turbulence and fluid structure interaction \cite{pod2,pod3,pod4}. A related method based on the balanced truncation technique of  \cite{moore}, and the snapshot POD technique, called the Balanced POD has been used to compute balancing transformations for large-scale systems \cite{willcox, rowley1}. The idea of Balanced POD is that by using  the impulse responses of both the primal and adjoint system, the most controllable and observable modes can be kept in the ROM. In 1978, Kung \cite{kung}  presented a new model reduction algorithm in conjuction with the singular value decomposition technique, and the Eigensystem Realization Algorithm (ERA) \cite{juang1985} was developed based on this technique. The BPOD is equivalent to the ERA procedure \cite{rowley4}, and forms the Hankel matrix using the primal and adjoint system simulations as opposed to the input-output data as in ERA. However, the advantage of the BPOD procedure is that the state of the full order system can be reconstructed from the ROM, as well as any non-zero initial condition projected into the reduced subspace of the ROM, something that is not achievable using the ROM obtained by ERA. More recently, there has been work on obtaining information regarding the dominant modes of a system, based on the snapshot POD followed by a diagonalization of the ROM matrix to extract the modes, called the dynamic mode decomposition (DMD) \cite{schmid, rowley2}.  Our method is a generalization of the DMD procedure, in that we randomly choose a suitable subset Hankel matrix from the full Hankel matrix, and show how adjoint information can be incorporated into the ROM, which leads to computational savings as well as more accurate results. We also provide error bounds on the eigenpairs resulting from the RPOD procedure. \\

The primary drawback of BPOD and ERA is that for a large scale system, such as that obtained by discretizing a PDE,  with a large number of inputs/outputs, the computational burden incurred is very high. There are two main parts to the computation, first is to collect datasets from both primal and adjoint simulation in order to generate the Hankel matrix. The second part is to solve the singular value decomposition problem for the resulting Hankel matrix. Thus, our primary goal in this paper is to reduce the computation required to obtain these ROMs without losing  accuracy. To this end, we introduce a randomized POD (RPOD) method which forms the ROM of a system using its dominant eigenmodes by solving the SVD problem of a suitably randomly chosen sub-matrix of the full Hankel matrix, and subsequently performing a diagonalization of the ROM to obtain the dominant modes. The computations required to form the sub-Hankel matrix, and the subsequent SVD, is computationally orders of magnitude less expensive when compared to the construction/ SVD of the full Hankel matrix, while providing almost the same information as the full Hankel matrix in terms of the numbers and accuracy of the underlying modes.  The RPOD is based on the BPOD and DMD, and retains the most controllable and observable modes in the ROM. We mention here that our ROM technique is SVD based and hence, different from Krylov subspace methods \cite{antoulas}. \\

There has been great interest in the Systems and Control community over the past several years in tractable randomized techniques to solve computationally difficult systems and control design problems \cite{RC1, RC2, RC3, RC4, RC5, RC6}. The RPOD technique can be construed as one such technique for the model reduction of large scale dynamical systems. In particular, it is perhaps most closely related to the ``Scenario Method" for systems and control design \cite{ RC5, RC6}. The scenario method obtains bounds on the number of convex constraints that need to be sampled from an uncountable set of constraints such that the solution to an associated robust control problem can be guaranteed to satisfy an $\epsilon$-fraction of the constraints, with probability greater than $1-\beta$ where $\epsilon, \beta$ are design parameters. In RPOD, we derive a bound for the total number of columns that need to be sampled from a low rank matrix (say rank $l$) containing a large number of columns,  given that the columns are spanned by modes $\{v_1, \cdots, v_l\}$, such that the sampled matrix has the same rank as the large matrix with probability at least $1-\beta$, given that the minimum fraction of the columns in which any of the spanning modes $v_i$ is present is $\bar{\epsilon}$. The scenario method obtains the bound $\frac{2}{\epsilon}(log(\frac{1}{\beta})+d)$ where $d$ is the dimension (size) of the problem, whereas our bound is $\frac{1}{\bar{\epsilon}} log(\frac{l}{\beta})$ where the rank $l$ is the size of our problem. Note the similarity between the two bounds except that our bound has the problem size  $l$ under the logarithm while the scenario method has the problem size $d$ outside the logarithm. The derivation of our bound, albeit different from the bound in \cite{RC5}, is nonetheless inspired by the developments in that reference. \\

We had introduced an iterative POD method (I-POD) in \cite{ACC, JAS} that recursively obtains eigenfunction of a linear operator using the individual input/output trajectories of the system. This paper shows that randomization of the procedure to choose a small subset of the input/output ensemble is sufficient  to extract all the relevant modes while increasing the accuracy and number of the extracted modes. Thus, the RPOD reduces the computation required to obtain the ROMs while at the same time, it increases the accuracy and number of the extracted modes.
The rest of the paper is organized as follows. In Section \ref{section 2}, we briefly introduce the POD and Balanced POD method, then show how to construct the eigenfunctions of the original system based on the snapshot POD technique. The eigenfunction reconstruction method using auto-correlation matrix between the input and output impulse responses respectively, and the cross-correlation matrix between the input and output impulse responses are introuduced as are error bounds on the reconstructed eigenvalues/ eigenvectors. In Section \ref{section 3}, we introduce the randomized proper orthogonal decomposition (RPOD) method where we randomly choose a subset of the inputs/outputs of the system to construct a sub-Hankel matrix when the number of inputs/outputs are large. Then we show that such an approximation contains the same information that is contained in the full Hankel matrix in terms of the dominant modes, given that the number of sampled inputs/ outputs satisfies a certain bound. Further, we compare the computational requirements of the RPOD method with BPOD/ ERA. In Section \ref{section 5}, we provide computational results comparing the RPOD with the BPOD for a 2 dimensional pollutant transport equation, a linearized channel flow problem, and the probability density  evolution in a 2 dimensional damped Duffing oscillator, governed by the Fokker-Planck-Kolmogorov equation.

\section{Eigenfunction reconstruction from Proper Orthogonal Decomposition(POD)}\label{section 2}
Consider a stable linear input-output system
\begin{eqnarray} \label{i/o}
x_k= Ax_{k-1} + Bu_k \nonumber\\
y_k = Cx_k,
\end{eqnarray}
where $x_k \in \Re^N$ is the state variable at discrete time instant $t_k$, $u_k \in \Re^p$ is a vector of inputs, and $y_k \in \Re^q$ is a vector of outputs. Let the input influence matrix be denoted by $B = [b_1,\cdots b_p]$ and the output matrix by $C = [c_1,\cdots c_q]'$.
The dimension of the state $N$ is very large. In the case of a PDE, the above system is obtained via a suitable discretization of the PDE using techniques such as finite Elements (FE)/ Finite Differences (FD).

In this section, first we briefly review the snapshot POD method and the Balanced POD method, then we introduce an eigenfunction reconstruction method based on the snapshot POD method. This method  reconstructs the eigenfunctions of the PDE operator that are present in input/output data, and uses them as a reduced order basis. This is done such that the reduced order basis, unlike in POD/BPOD,  is independent of the data that is used to construct the reduced order model. It also helps us in distingushing underlying invariant modes when we implement the RPOD algorithm, introduced in section \ref{section 3}, in a recursive fashion.

\subsection{Preliminaries}
Consider the linear system (\ref{i/o}), first, we introduced the snapshot POD method proposed by Sirovich in \cite{pod4}.

If we collect the data at timesteps $t_1, t_2, \cdots, t_{M_1}$ during time $0 \leq t \leq T$, and denote the data as $X=[x(t_1), x(t_2), \cdots, x(t_{M_1})]$. Then the POD method seeks to find a projection $P_r$ which can minimize the error
\begin{eqnarray}
\sum_{k=1}^{k=M_1} \|x(t_k)-P_rx(t_k)\|^2.
\end{eqnarray}

To solve this minimization problem, we need to solve the eigenvalue problem:
\begin{eqnarray}
(X^*X)V = \Lambda V,
\end{eqnarray}
where $X^*$ denotes the transpose of $X$, $(\Lambda, V)$ are the non-zero eigenvalues and the corresponding eigenvectors of $X^*X$. Then the POD projection can be constructed as:
\begin{eqnarray}
P_r = X V\Lambda^{-\frac{1}{2}}.
\end{eqnarray}

Thus, the reduced order model constructed using the snapshot POD method is:
\begin{eqnarray}
A_r = P_r'AP_r
\end{eqnarray}

Next, we introduce the Balanced POD method \cite{willcox, rowley1} using the impulse response of the primal and adjoint system.

We collect the impulse response of the primal system by using $b_j$, $j = 1, 2, \cdots, p$, as initial conditions for the simulation of the system, 
\begin{eqnarray} \label{dps}
x_k^{(j)}= Ax_{k-1}^{(j)},
\end{eqnarray}

If we take $M_1$ snapshots across the trajectories at time $t_1, t_2, \cdots, t_{M_1}$, resulting an $N \times pM_1$ matrix
\begin{eqnarray}
X= [X_1, X_2, \cdots, X_p],
\end{eqnarray}
where $X_j = [X_j(t_1), X_j(t_2), \cdots, X_j(t_{M_1})] $ is the state $x$ at time instant $t_1, t_2, \cdots, t_{M_1}$ from the $j^{th}$ input trajectory.

Similarly, we use the transposed rows of the output matrix, $c_i'$, as the initial conditions for the simulations of the adjoint system $A'$, and $M_2$ snapshots are taken across trajectories, leading to the adjoint snapshot ensemble $Y$,
\begin{eqnarray}
Y= [Y_1, Y_2, \cdots, Y_q],
\end{eqnarray}
where $Y$ is an $N \times qM_2$ matrix and $Y_i =[Y_i(\hat{t}_1), Y_i(\hat{t}_2), \cdots, Y_i(\hat{t}_{M_2})]$ is the output $y$ at time instant $\hat{t}_1, \hat{t}_2, \cdots, \hat{t}_{M_2}$ from the $i^{th}$ output trajectory,  $i=1,2, \cdots, q$.

The Hankel matrix $H$ constructed using the input influence matrix at timesteps $(t_1, t_2,\cdots, t_{M_1})$ and the output influence matrix at timesteps $(\hat{t}_1, \hat{t}_2,\cdots, \hat{t}_{M_2})$ is:
\begin{eqnarray}
H= Y'X=  \nonumber \\
\left( \begin{array}{cccc}
CA^{(t_1+\hat{t}_1)}B & CA^{(t_2+\hat{t}_1)}B & \cdots & CA^{(t_{M_1}+\hat{t}_1)}B \\
CA^{(t_1+\hat{t}_2)}B &CA^{(t_2+\hat{t}_2)}B &\cdots& CA^{(t_{M_1}+\hat{t}_2)}B\\
\hdots& \hdots& \cdots & \hdots\\
CA^{(t_1+ \hat{t}_{M_2})}B &CA^{(t_2+\hat{t}_{M_2})}B & \cdots & CA^{(t_{M_1}+\hat{t}_{M_2})}B
\end{array} \right). \label{hankel}
\end{eqnarray}

Then we solve the singular value decomposition (SVD) problem of the matrix $H$:
\begin{eqnarray}
H=Y'X = U \Sigma V'.
\end{eqnarray}

Assume that $\Sigma_1$ consists of  the first $r$ non-zero singular values of $\Sigma$, and $(U_1, V_1)$ are the corresponding left and right singular vectors from $(U,V)$, then the POD projection matrices can be defined as:
\begin{eqnarray}
T_r = XV_1 \Sigma_1^{-\frac{1}{2}}, \nonumber\\
T_l = YU_1 \Sigma_1^{-\frac{1}{2}},
\end{eqnarray}
and the reduced order model constructed using BPOD method is:
\begin{eqnarray}
\begin{cases}
A_r = T_l' A T_r \\
B_r = T_l' B \\
C_r = C T_r
\end{cases}
\end{eqnarray}
\subsection{Eigenfunction reconstruction using auto-correlation matrix}

From the previous section, we can see that the reduced order model constructed using snapshot POD method and BPOD method are not invariant to the datasets $X$ and $Y$. When the collected snapshots $X$ and $Y$ are changed, the POD bases $T_r$ and $T_l$ change too. Thus, we want to construct a global reduced order model which remains invariant to the particular primal and adjoint simulation snapshots $X$ and $Y$. First, we show how to reconstruct the eigenfunctions of the original system based on the POD method, and then we construct the ROM from the extracted eigenfunctions, which by definition is then invariant to the data.

Suppose we use the same impluse response of the primal and adjoint system as above. Following the snapshot POD procedure, we can get the POD basis $T_r$ of the trajectory encoded in the snapshot ensemble $X$ as follows:
\begin{eqnarray} \label{snapshot_POD}
T_r = XV_p\Sigma_p^{-1/2},
\end{eqnarray}
where $\Sigma_p$ are the first $n$ non-zero eigenvalues of the correlation matrix $X'X$, and $V_p$ are the corresponding eigenvectors, i.e.,
\begin{eqnarray}
(X'X)V_p = V_p \Sigma_p.
\end{eqnarray}

Given the snapshot POD eigenfunctions, we can obtain a reduced order approximation of the system in (\ref{dps}) as follows:
\begin{eqnarray}
\psi_k= (T_r'AT_r)\psi_{k-1} \equiv \tilde{A} \psi_{k-1}, 
\end{eqnarray}
where $\psi$ represents the projection of the system state onto the POD eigenfunctions and $\tilde{A}$ represents the reduced order $n  \times n$ system matrix.

Assume that $\tilde{A}$ has a full set of distinct eigenvectors. Let $(\Lambda_r, P_r)$ represent the eigenvalue-eigenvector pair for $\tilde{A}$, i.e.,
\begin{eqnarray}
\tilde{A}P_r = P_r\Lambda_r.
\end{eqnarray}

Noting that $\tilde{A} = P_r\Lambda_r P_r^{-1}$, the ROM matrix $\tilde{A}$ transformed to the co-ordinates specified by $P_r$, can be represented in the modal co-ordinates $\phi$ as:
\begin{eqnarray}
\phi_k= \Lambda_r \phi_{k-1}.
\end{eqnarray}

Thus it follows that
\begin{eqnarray} \label{RI-POD_key}
\Lambda_r = V_r^{-1}A V_r,
\end{eqnarray}
where $V_r=T_rP_r$. Here, $T_r$ is the POD transformation basis and $P_r$ is the ROM eigenfunction matrix. Note that $T_r$ is $N \times n$ and that $P_r$ is $n \times n$, and hence, $V_r$  is $N \times n$. The transformation $V_r$ denotes the composite transformation from the original state space to the POD eigenfunction space, and in turn to the ROM eigenfunction space.

Similarly, we can get the POD basis $T_l$ using the adjoint simulation ensemble $Y$
\begin{eqnarray} \label{ad_snapshot_POD}
T_l = YU_p\hat{\Sigma}_p^{-1/2},
\end{eqnarray}
where $U_p$ and $\hat{\Sigma}_p$ are the eigenvector-eigenvalue pair corresponding to the correlation matrix $Y'Y$. If $(\Lambda_l, P_l)$ represent the eigenvalue-eigenvector pair for reduced order model $\hat{A}= T_l^{-1}A'T_l$. Here, suppose we use the $m^{th}$ order approximation, i.e., $\hat{\Sigma}_p$ are the first $m$ non-zero eigenvalues of $Y'Y$, thus, we have that $V_l = T_lP_l$ is an $N \times m$ matrix. 

In the following, we relate the eigenvalues and right eigenvectors of $A$ to the diagonal form $\Lambda_r$ and the transformation $V_r$.

\begin{assumption} \label{A1}
Assume that there are at most ``$n$" eigenvectors of the matrix $A$ active in the snapshot ensemble $X = [X_1, X_2, \cdots, X_p]$, i.e., 
\begin{eqnarray}
X_i(t_k)= \sum_{j=1}^n \alpha_j^i(t_k) v_j, i=1,2,......p
\end{eqnarray}
where $v_j$ is the eigenvector of $A$. We assume that $n \leq pM_1$, which means that the number of the active modes in the snapshots should be less than or equal to the total number of the snapshots.
\end{assumption}
\begin{remark}
The rank of the snapshots is $min(n, pM_1)$, thus, we require $n \leq min(n,pM_1)$ to make sure that the data is overdetermined in terms of the underlying modes. This assumption can be guaranteed by taking enough snapshots. If different eigenvectors are active in different trajectories, then we take the union of these active eigenvectors and denote the total number of the active eigenvectors by $n$.
\end{remark}

Under Assumption \ref{A1}, the following result is true.

\begin{proposition}\label{P1}
The eigenvalues of the ROM $\tilde{A}$, given by the diagonal matrix $\Lambda_r$, are eigenvalues of the full order model $A$,  and the corresponding right eigenvectors are given by the transformation $T_rP_r$.
\end{proposition}

\begin{proof}
Recall that $T_r = XV_p\Sigma_p^{-1/2}$. We have
\begin{eqnarray}
X = V\mathbf{\alpha} = [v_1, v_2,\cdots v_n]\begin{bmatrix} \alpha_1^ 1(t_1) & .. & \alpha_1^p(t_{M_1})) \\. & .. & . \\\alpha_n^1(t_1) & .. & \alpha_n^p(t_{M_1}) \end{bmatrix}, \nonumber
\end{eqnarray}
where $V$ denotes the active right eigenvectors of $A$ in the snapshots, and $\alpha$ is the coefficient matrix of the eigenvectors for all the snapshots, note that $V$ is an $N \times n$ and $\mathbf{\alpha}$ is an $n \times pM_1$ matrix. Thus, $X'X \in R^{pM_1 \times pM_1}$, and has $pM_1$ eigenvalues. From Assumption \ref{A1}, the number of acitve modes is $n$, and $n \leq pM_1$. Assume the number of non-zero eigenvalues of $X'X$ is $r$, where $r \leq n$. First, we prove the case when $r=n = pM_1$, which means we keep the $n^{th}$ order approximation. Then, it follows that
\begin{align}
\tilde{A} = T_r' A T_r = \Sigma_p^{-1/2}V_p'\alpha'V'AV\alpha V_p\Sigma_p^{-1/2} \nonumber\\
=  \underbrace{\Sigma_p^{-1/2}V_p'\alpha'V'V }_{\tilde{P}'(V'V)}\tilde{\Lambda}\underbrace{\alpha V_p\Sigma_p^{-1/2}}_{\tilde{P}}
= P_r\Lambda_r P_r^{-1},
\end{align}
where $\tilde{\Lambda}$ are the eigenvalues of $A$ corresponding to the eigenvectors $V$.  Thus, if we show that $\tilde{P}$ is the inverse of $\tilde{P}'(V'V)$, then due to the uniqueness of the similarity transformation of $\tilde{A}$, it follows that $\tilde{P} = P_r^{-1}$ and $\Lambda_r = \tilde{\Lambda}$. To show this, note that:
\begin{align}
\tilde{P}'(V'V)\tilde{P} = \Sigma_p^{-1/2} V_p' \alpha'(V'V) \alpha V_p \Sigma_p^{-1/2}.
\end{align}

Here $\alpha'(V'V) \alpha = X'X = V_p\Sigma_pV_p'$, and therefore, using the orthogonality of the columns of $V_p$, it follows that
\begin{align}
\tilde{P}'(V'V)\tilde{P} = \Sigma_p^{-1/2} V_p' V_p \Sigma_p V_p'V_p \Sigma_p^{-1/2} = I.
\end{align}

Hence, $\tilde{P}$ and $\tilde{P}'(V'V)$ are inverses of each other. It follows that:
\begin{align}
T_r P_r = XV_p\Sigma_p^{-1/2} \Sigma_p^{-1/2}V_p'\alpha'V'V \nonumber \\
= V (\alpha V_p\Sigma_p^{-1/2})( \Sigma_p^{-1/2}V_p'\alpha'V'V) = V \underbrace{\tilde{P} \tilde{P}' (V'V)}_{I}= V
\end{align}

i.e., the columns of $T_rP_r $ are indeed right eigenvectors of A. Moreover, it also follows that owing to the uniqueness of the similarity transformation $\tilde{A}$ that the eigenvalues corresponding to the eigenvectors in $T_rP_r$ are in the diagonal form $\Lambda_r$. Hence, this proves our assertion for the case when we keep the $n^{th}$ order approximation.

If $r < n$, then, the transformation into the POD basis $T_r = XV_p\Sigma_p^{-1/2}$ should only include the POD eigenvectors corresponding to the $r$ non-zero eigenvalues. This implies that $X'X = \alpha'V'V\alpha =  \hat{V}_p \hat{\Sigma}_p\hat{V}_p'$, where $\hat{\Sigma}_p$ contains the $r$ non-zero POD eigenvalues, and $\hat{V}_p$ contains the corresponding eigenvectors. The analysis above goes through unchanged, and hence, $\hat{P}\hat{P}'V'V = I$, and $T_rP_r = V$, where $V$ now consists of the $r$ active eigenvectors.  
\end{proof}

Next, we want to discuss the errors resulting from the fact that Assumption \ref{A1} cannot be exactly satisfied. If we denote $V=[v_1, v_2,\cdots, v_N]$ as the right eigenvectors of system matrix $A$, and
$\alpha = \begin{bmatrix} \alpha_1^ 1(t_1) & .. & \alpha_1^p(t_{M_1}) \\. & .. & . \\\alpha_N^1(t_1) & .. & \alpha_N^p(t_{M_1}) \end{bmatrix}$ as the coefficient matrix, from Assumption \ref{A1}, we need that $\alpha_{n+1}^j,\alpha_{n+2}^j,...,\alpha_N^j=0$, $j = 1, 2, \cdots, p$. However, $\alpha_{n+1}^{j},\alpha_{n+2}^j,...,\alpha_N^j \approx 0$. Thus, we need to characterize the errors from the fact that these coefficients are near zero and not exactly zero. Denote
\begin{eqnarray}
X_{id}= \left(\begin{array}{cc} V_N & V_E\end{array}\right) \left(\begin{array}{c} \alpha_N\\ 0 \end{array}\right)=V_N \alpha_N, \nonumber \\
X_{ac}= \left(\begin{array}{cc} V_N & V_E\end{array}\right) \left(\begin{array}{c} \alpha_N\\ \delta \alpha \end{array}\right)=X_{id}+ V_E \delta \alpha
\end{eqnarray}
Here, $X_{id}$ is the ideal snapshots required in Assumption \ref{A1}, while $X_{ac}$ is the actual set of snapshots, and we assume that $\|\delta \alpha \| \leq C \epsilon$. With this assumption, we have the following result.

\begin{proposition} \label{P2}
Assume that both A and $\tilde{A}$ have distinct set of eigenvalues, $\delta \alpha$ is the coefficient matrix which is defined above, and $\| \delta \alpha \| \leq C \epsilon$, where $C$ is a constant, and $\epsilon$ is sufficiently small.
Then the errors resulting from Assumption \ref{A1} not being exactly satisfied result in the following errors in the reconstruction eigenvalue and eigenvectors:
$\| \Lambda- \hat{\Lambda} \| \leq k_1 \epsilon$, where $\hat{\Lambda}$ is the diagonal matrix contains the actual eigenvalues of the system matrix A which are active in the snapshots $X$, and $\| V_r- \hat{V}  \| \leq k_2 \epsilon$, where $\hat{V}$ is the set of corresponding actual eigenvectors of system matrix A.
\end{proposition}

The proof of this proposition uses the eigenvalue perturbation theory \cite{eigp} and the eigenfunction reconstruction technique introduced above. The proof  is shown in Appendix \ref{appendix A}.

\begin{remark}
Since the left eigenvectors of $A$ are found by using the adjoint system $A'$, and the right eigenvectors of $A'$ are the same as the left eigenvectors of $A$,  Proposition \ref{P1} and \ref{P2} hold for the left eigenvectors of $A$ as well.
\end{remark}

We have the right eigenvalue-eigenvector pair $({\Lambda}_r, V_r) $ from the snapshot ensemble $X$, and the left eigenvalue-eigenvector pair $({\Lambda}_l, V_l) $ from the adjoint simulation snapshots $Y$. Among these eigenpairs, we only keep those left/ right eigenvectors that corresponding to the eigenvalues in the intersection of $\Lambda_l$ and $\Lambda_r$.

Then the reduced order model of (\ref{i/o}) is:
\begin{eqnarray}
\psi_{k} = (V_l'AV_r) \psi_{k-1} + V_l'Bu_k, \; \psi_i(0) = (x(0), v_{li}) \nonumber\\
y_k = CV_r \psi_{k-1}
\end{eqnarray}

\begin{remark}
We should note that, theoretically, the transformation $V_r$ and $V_l$ are the right and left eigenvectors of system matrix A, however, practically, $V_r$ and $V_l$ may not be orthogonal to each other, which may cause inaccuracy, and even instablility of the reduced order system, so we need to add a biorthogonalization algorithm. Here, we use a two-sided modified Gram-Schmidt process to re-biorthogonalize the set. The method is shown below:

for $i=1,2,\cdots, j $
\begin{eqnarray}
V_l^{j+1} = V_l^{j+1} - V_l^{j}((V_r^{j} )^{H}V_l^{j+1}) \nonumber \\
V_r^{j+1}= V_r^{j+1} - V_r^{j}((V_l^{j})^{H} V_r^{j+1}) 
\end{eqnarray}

end for
\end{remark}
\subsection{Eigenfunction reconstruction using cross-correlation matrix}
In practice, the results using the cross-correlation between the output trajectories Y and input trajectories X are better than the method outlined using the auto-correlation matrix. By using the cross-correlation matrix, biorthogonality of the bases $T_r$ and $T_l$ is guaranteed, and we can also save computations needed to match the left and right eigenvector pairs. Further, the eigenpairs reconstruction is much more accurate. This method is used in all the computational results reported in this paper. 

 We form the right POD basis $T_r = XV_p \Sigma_p^{-1/2}$, and the left POD basis $T_l = YU_p \Sigma_p^{-1/2}$, which are the same as the auto-correlation case, but here $(U_p, \Sigma_p, V_p)$ is the solution of the singular value decomposition problem: 
\begin{eqnarray}
H=Y'X = U_p\Sigma_p V_p'.
\end{eqnarray}

We have to solve the eigenvalue problem of $\tilde{A}$:
\begin{eqnarray}
\tilde{A} = (\Sigma_p^{-1/2} U_p'Y')A(XV_p \Sigma_p^{-1/2}) = P\Lambda_{ij} P^{-1} .
\end{eqnarray}

The reduced order model is:
\begin{eqnarray}
\begin{cases}
A_r = \Lambda_{ij} = \underbrace{(P^{-1}\Sigma_p^{-1/2}U_p'Y')}_{\Phi_{ij}'}A\underbrace{(X V_p \Sigma_p^{-1/2} P)}_{\Psi_{ij}}\\
B_r = \Phi_{ij}'B\\
C_r = C \Psi_{ij}
\end{cases}\label{RPODROM}
\end{eqnarray}

Notice that the active left eigenvectors in the snapshots $Y$ and the active right eigenvectors in the snapshots $X$ may not be the same. We assume that the contribution of the left and right eigenvectors corresponding to the different eigenvalues are small. Thus, we denote $X_{ac} = V_S \alpha_S + V_D \delta \alpha_D$, $Y_{ac} = U_S \beta_S + U_D \delta \beta_D$, where $(U_S, V_S)$ are the active left and right eigenvectors corresponding to the same eigenvalues $\Lambda_S$, $(U_D, V_D)$ are the rest of the left and right eigenvectors.    
We assume that $\| \delta \alpha_D \| \leq C_1 \epsilon$, where $C_1$ is some constant, and $\epsilon$ is sufficiently small. Similarly, $\| \delta \beta_D \| \leq C_2 \epsilon$. The following result then holds.

\begin{proposition} \label{P3}
Denote $(\Lambda_S, U_S, V_S)$ as the actual eigenvalues, left and right eigenvectors of $A$ which are active in both sets of snapshots $X$ and $Y$.
Under the assumption that $\| \delta \alpha_D \| \leq C_1 \epsilon$,  $\| \delta \beta_D \| \leq C_2 \epsilon$, for sufficiently small $\epsilon$,
the errors in eigenvalue and eigenvector reconstruction using the cross-correlation matrix are 
$\| \Lambda_{ij}- \Lambda_S\| \leq k_1 \epsilon^2$, $\| \Phi_{ij}- U_S\| \leq k_2 \epsilon$, and $\| \Psi_{ij}- V_S\| \leq k_3 \epsilon$, i.e., $\Lambda_{ij}$, $\Phi_{ij}$ and $\Psi_{ij}$ are arbitrarily good approximation of the eigenvalues, left and right eigenvectors active in both sets of snapshots $X$ and $Y$, given that $\epsilon$ is sufficiently small.
\end{proposition}

The proof is shown in Appendix \ref{appendix A}. We can see that $\Lambda_{ij}$ contains the most observable and controllable eigenmodes present in the adjoint/ primal response data Y and X. Also note that the eigenvalues extracted using the cross-correlation matrix are much more accurate than using the auto-correlation matrix ($O(\epsilon^2)$ vs $O(\epsilon)$). Further, the left/ right eigenvectors constructed are orthogonal by construction.
\section{Randomized Proper Orthogonal Decomposition Method}\label{section 3}
From Section \ref{section 2}, we see that we can construct POD bases, and extract the underlying eigenvectors of the original system, which are invariant to the particular primal and adjoint datasets $X$ and $Y$. Assume that the rank of the full Hankel matrix $H = Y'X$ is $l$. Since the dimension of the systems governed by PDEs may be very large due to the discretization, the computation to construct the Hankel matrix and solve the SVD problem is very expensive, especially when there are a large number of inputs/ outputs. The eigenfunction reconstruction technique from Section \ref{section 2} suggests that if we can construct a sub-Hankel matrix $\hat{H}$ which is still rank $l$, then the underlying $l$ eigenmodes can be recovered from the sub-Hankel matrix. Thus, we introduce a randomized proper orthogonal decomposition(RPOD) method based on the eigenfunction reconstruction technique which randomly chooses a small subset of the inputs/ outputs, and constructs a sub-Hankel matrix from the full Hankel matrix such that the information encoded in the sub-Hankel matrix is almost the same as that in the full Hankel matrix, in terms of the number and accuracy of the underlying modes that can be extracted. In Section \ref{err}, we show how to randomly choose the inputs/ outputs from the original system, and show that the RPOD method extracts exactly the same information, in terms of the dominant modes, from a much smaller sub-Hankel matrix as can be extracted from the full Hankel matrix. In Section \ref{section 4}, we compare the computational requirements of RPOD and BPOD.

\subsection{The RPOD Technique}\label{err}
\begin{algorithm}[!tb] \label{RPOD_algo}
\begin{enumerate}
\item{For $i=1 : r$, for $j=1 : s$}
\item{Pick $c_i \in \{1, \cdots, p\}$ with probability $P[c_i = k] = \frac{1}{p}, k=1, \cdots, p$}
\item{Pick $r_j \in \{1, \cdots, q\}$ with probability $P[r_j = k] = \frac{1}{q}, k = 1, \cdots, q$}
\item{Set $\hat{B}^{(i)} = B^{(c_i)}$, $\hat{C}_{(j)} = C_{(r_j)}$}
\item{Using $\hat{B}^{(i)}, i= 1, \cdots, r$ as the initial conditions for the primal simulation, collect the snapshots at $t=\tilde{t}_1, \cdots, \tilde{t}_{m_1}$, denoted as $\hat{X}$}
\item{Using $(\hat{C}_{(j)})^T, j = 1, \cdots, s$ as the initial conditions for the adjoint simulation, collect the snapshots at $t = \tilde{t}_1, \cdots, \tilde{t}_{m_2}$, denoted as $\hat{Y}$}
\item{Construct the reduced order Hankel matrix $\hat{H} = \hat{Y}^T \hat{X}$}
\item{Solve the SVD problem of $\hat{H} = U_p \Sigma_p V_p$}
\item{Construct the POD basis: $T_r = \hat{X}V_p \Sigma_p^{-1/2}$, $T_l = \hat{Y}^TU_p^T \Sigma_p^{-1/2}$}
\item{Construct the matrix: $\tilde{A} = T_l^T A T_r$, and $(\Lambda, P)$ are the eigenvalues and eigenvectors of $\tilde{A}$}
\item{Construct new POD basis: $\Phi'= P^{-1}T_l'$ and $\Psi = T_r P$}
\item{The ROM is: $A_r= \Phi' A \Psi$, $B_r = \Phi' B$, $C_r = C \Psi$}
\end{enumerate}
\caption{RPOD Algorithm}
\end{algorithm}

Consider the stable linear system (\ref{i/o}), we randomly choose $r$ columns from $B$ according to the uniform distribution, denoted as $\hat{B}$, and randomly choose $s$ rows from $C$ with uniform distribution, denoted as $\hat{C}$. Denote $(.)^{(i)}$ as the column of $(.)$, and $(.)_{(i)}$ as the rows of $(.)$, then the RPOD procedure is summarized in Algorithm \ref{RPOD_algo}.

Define matrix $p_1$, $p_2$ such that 
\begin{eqnarray}
\hat{B} = Bp_1, \nonumber\\
\hat{C} = p_2C,
\end{eqnarray}
where $p_1 \in R^{p \times r}$, $p_1^{(i,j)} =1, i= 1, 2,...,p, j = 1, 2, ..., r$ if the $i^{th}$ column of $B$ is chosen, and $p_1^{(i,j)} = 0$ otherwise. Similarily, $p_2 \in R^{s \times q}$, where $p_2^{(i,j)} = 1, i= 1, 2,..., s, j=1, 2, ..., q$ if the $j^{th}$ row of $C$ is chosen, and $p_2^{(i,j)} = 0$ otherwise. 
The original Hankel matrix $H$ constructed using the input influence matrix at timesteps $(t_1, t_2, \cdots, t_{M_1})$ and the output influence matrix at timesteps $(\hat{t}_1, \hat{t}_2, \cdots, \hat{t}_{M_2})$ was previously defined in (\ref{hankel}). The reduced order Hankel matrix $\hat{H}$ is then constructed using $\hat{B}$, $\hat{C}$ at timesteps $(\tilde{t}_1, \tilde{t}_2, \cdots, \tilde{t}_{m_1})$ and timesteps $(\tilde{t}_1, \tilde{t}_2, \cdots, \tilde{t}_{m_2})$ respectively. Here, $(\tilde{t}_1, \tilde{t}_2, \cdots, \tilde{t}_{m_1})$ are randomly chosen from the timesteps $(t_1, t_2, \cdots, t_{M_1})$ with uniform distribution, and $(\tilde{t}_1, \tilde{t}_2, \cdots, \tilde{t}_{m_2})$ are randomly chosen from the timesteps $(\hat{t}_1, \hat{t}_2, \cdots, \hat{t}_{M_2})$ with uniform distribution.  Thus, the RPOD technique can be seen as randomly choosing $rm_1$ columns from the $H$ matrix to form the $\tilde{H}$ matrix, and then randomly choosing $sm_2$ rows from the $\tilde{H}$ matrix to form $\hat{H}$. Alternatively, it essentially is equivalent to choosing a suitable random subset of the columns of the primal/ adding responses, namely $\hat{X}$ and $\hat{Y}$ to generate the sub-Hankel matrix $\hat{H} = \hat{Y}'\hat{X}$.\\

First, we provide a general result regarding randomly choosing a rank $``l"$ sub-matrix from a large rank $``l"$ matrix. Suppose $X \in R^{P \times Q}$ is a rank $l$ matrix, and suppose that $X$ is spanned by the vectors $\{v_1, v_2, \cdots v_l\}$, $l\ll P, Q$. Let $X^{(i)}$ denote the set of columns of $X$ that contain the vector $v_i$.
Let
\begin{align}
\epsilon_i = \frac{n(X^{(i)})}{N},
\end{align}
denote the fraction of the columns in $X$ in which vector $v_i$ is present. Further let 
\begin{align}
\bar{\epsilon} = \min_i \epsilon_i,
\end{align} 
and note that $\bar{\epsilon} > 0$. 
\begin{proposition} \label{Mbound}
Let $M$ columns be sampled uniformly from among the columns of the matrix $X$, and denote the sampled sub-matrix by $\hat{X}$.  Let $(\Omega, \mathcal{F}, P)$ denote the underlying probability space for the experiment.  Given any $\beta > 0$, if the number $M$ is chosen such that
\begin{align}
M > max(l, \frac{1}{\bar{\epsilon}} log(\frac{l}{\beta})),
\end{align}
then $P(\rho(\hat{X}) < l) < \beta$, where $\rho(\hat{X})$ denotes the rank of the sampled matrix $\hat{X}$.
\end{proposition}
\begin{proof}
Let $\hat{X} (\omega) = \{X_{1}(\omega), \cdots X_{M}(\omega)\}$ denote a random M-choice from the columns of $X$. If the ensemble $\hat{X}$ has rank less than $l$ then note that atleast one of the vectors $v_i$ has to be absent from the ensemble. Define the events
\begin{align}
B = \{\omega \in \Omega : \rho(\hat{X}(\omega)) < l\}, \, \mbox{and}\\
B_i = \{\omega \in \Omega:{X_{k}}(\omega) \in \tilde{X}^{(i)}, \forall k \},
\end{align}
where $\tilde{X}^{(i)}$ denotes the complement set of columns in $X$ to the set $X^{(i)}$. Due to the fact that the ensemble $\hat{X}$ is rank deficient if all of the columns of $\hat{X}$ are sampled from atleast one of the sets $\tilde{X}^{(i)}$, and the fact that if $\hat{X}$ is rank deficient, all the columns of $\hat{X}$ have to be sampled from at least one of the sets $\tilde{X}_i$, it follows that:
\begin{align}
B = \bigcup_i B_i.
\end{align}
However, $P(B_i) \leq (1-\epsilon_i)^M$. Thus, it follows that
\begin{align} 
P(B) \leq \sum_{i=1}^l P(B_i) = \sum_{i=1}^l (1-\epsilon_i)^M, \nonumber\\
\leq l(1-\bar{\epsilon})^M. \label{bound_key}
\end{align}
Hence, it follows that $P(\rho(\hat{X}) < l) \leq l(1 -\bar{\epsilon})^M$. If we require this probability to be less than some given $\beta >0$, then, it can be shown by taking log on both sides sides of the above expression that $M$ should satisfy
\begin{align} \label{Meff}
M > \frac{1}{\bar{\epsilon}} log(\frac{l}{\beta}).
\end{align}
Noting that $\hat{X}$ is rank deficient unless $M \geq l$, the result follows.\\
\end{proof}
\begin{remark}
 \textit{Effect of $l, \bar{\epsilon}$ on the bound M:} 
 It can be seen that the number of choices $M$ is influenced primarily by $\bar{\epsilon}$ and not significantly by the number of active modes/ rank of the ensemble $l$, since $l$ appears in the bound under the logarithm. Thus, the difficulty of choosing a sub-ensemble that is rank $l$ is essentially decided by the fraction $\bar{\epsilon_i}$ of the ensemble in which the rarest vector $v_i$ is present. Moreover, note that as the number $l$ increases, we need only sample $\mathcal{O}(l)$ columns to have a rank $``l"$ sub-ensemble.\\
\end{remark}
\begin{remark}
\textit{Effect of Sampling non-uniformly:} 
In certain instances, for instance, when we have a priori knowledge, we may choose to sample the columns of $X$ non-uniformly. Define 
\begin{align}
\epsilon_i^{\Pi} = \sum_{j=1}^N 1_i(X_j)\pi_j,
\end{align}
where $\pi_j$ is the probability of sampling column $X_j$ from the ensemble $X$, and $ 1_i(X_j)$ represents the indicator function for vector $v_i$ in column  $X_j$, i.e, it is one if $v_i$ is present in $X_j$ and 0 otherwise. Note that $\epsilon_i$ as defined before is the above quantity with the uniform sampling distribution $\pi_j= \frac{1}{N}$ for all $j$. It is reasonably straightforward to show that Proposition \ref{Mbound}  holds with $\bar{\epsilon}^{\Pi} = \min_i \bar{\epsilon_i}^{\Pi}$ for any sampling distribution $\Pi$ ( we replace $\epsilon_i$ in Eq. \ref{bound_key} with $\bar{\epsilon}_i^{\Pi}$) . The effect of a good sampling distribution is to lower the bound $M$ by raising the number $\bar{\epsilon}^{\Pi}$ over that of a uniform distribution. This may be an intelligent option when otherwise the bound on $M$ with uniform sampling can be very high, for instance when one of the vectors $v_i$ is present in only a very small fraction of the ensemble $X$. However, we might have some a priori information regarding the columns where $v_i$ may be present and thus, bias the sampling towards that sub-ensemble. \\
\end{remark}

Next, it can be seen how the RPOD procedure extends the above result to the Balanced POD scenario where we consider the Hankel matrix $H = Y'X$, where $H$ is of size $qM_2 \times pM_1$. Suppose again that the output and input ensembles $Y$ and $X$ are spanned by the same set of left/ right eigenvectors $U = \{u_1, \cdots u_l\}$ and $V = \{v_1, \cdots v_l\}$  respectively, corresponding to the same set of eigenvalues $\Lambda = \{\lambda_1, \cdots \lambda_l\}$.  Thus, the Hankel matrix $H$ is rank $l$. 
Define:
\begin{align} \label{eps_def}
\bar{\epsilon}_X = \min_i \epsilon_{X,i},\nonumber\\
\bar{\epsilon}_Y = \min_j \epsilon_{Y,j},
\end{align} 
where $\epsilon_{X,i}$ is the fraction of columns in $X$ in which the right eigenvector $v_i$ is present, and $\epsilon_{Y,j}$ is the fraction of the columns in $Y$ in which the left eigenvector $u_j$ is present.\\
The RPOD  chooses a small number of inputs/ outputs, namely $s$/$r$ respectively, and then chooses a small number  of times, $m_1$ for the input and $m_2$ from the outputs, at which to sample the input/ output trajectories, and form the sub-Hankel matrix  $\hat{H}$ which is much smaller in size, $sm_2 \times rm_1$ when compared to the original Hankel matrix $H$.  This is equivalent to a uniform sampling of the columns of the input and output ensembles $X$ and $Y$ respectively to form $\hat{H} = \hat{Y}'\hat{X}$. Note that due to Proposition \ref{Mbound},  given any $\beta>0$, if we choose the number of inputs/ outputs $r/s$, and the timesteps at which to sample these input/ output trajectories $m_1/ m_2$, in such a way that $rm_1$ and $sm_2$  satisfy the bounds:
\begin{align} \label{RPOD_bound}
rm_1> max(l, \frac{1}{\bar{\epsilon}_X}log(\frac{l}{\beta})),  \nonumber\\
sm_2 > max(l, \frac{1}{\bar{\epsilon}_Y}log(\frac{l}{\beta})), 
\end{align}
then the probability of $\hat{H}$ having rank less than $l$ is less than $\gamma = 1-(1-\beta)^2$, since then the probability that the ranks of the sampled input and output ensembles are less than $l$, is less than $\beta$. Thus, if we repeatedly choose $K$ such ensembles with replacement, the probability of having a sub-Hankel matrix $\hat{H}$ that is still less than rank $l$ after the $K$ picks, has to be less than $\gamma^K$. Thus, the probability of choosing a rank $l$ sub-Hankel matrix $\hat{H}$ exponentially approaches unity with the number of trials. Again, noting that the value of $\beta$ does not have a significant influence on the bounds above, it follows that $\beta$ can be chosen to  be quite high without significantly affecting the number of columns that need to be chosen to satisfy the confidence level of $\beta$, and thus, the probability of choosing a rank $l$ sub-Hankel matrix can be made arbitrarily high by judiciously choosing the number of columns in the input/ output ensembles according to the bounds in Eq. \ref{RPOD_bound}. We summarize the development above in the following proposition.\\

\begin{proposition}\label{RPOD_prop}
Let $H = Y'X$ be a $qM_2 \times pM_1$ Hankel matrix with $p$ inputs, $q$ outputs, $M_1$ time snapshots in every input trajectory and $M_2$ time snapshots for every output trajectory. Let the left/ right eigenvectors $ U=\{u_1, \cdots u_l\}$, and $V = \{v_1, \cdots v_l\}$ denote the eigenvectors spanning the input and output ensembles  $X$ and $Y$ respectively. Let $\bar{\epsilon}_X, \bar{\epsilon}_Y$ be as defined in Eq. \ref{eps_def}. Let $\beta> 0$ be given. Suppose we construct a sub-Hankel matrix $\hat{H}$ according to the RPOD procedure:  by uniformly sampling $r$ inputs with $m_1$ time snapshots, and $s$ outputs with $m_2$ snapshots, and that $rm_1$ and $sm_2$ are chosen as in Eq. \ref{RPOD_bound}, then the probability that the sub-Hankel matrix has rank less than $l$ is less than $\gamma = 1-(1-\beta)^2$. Moreover, the probability that after $K$ RPOD choices, with replacement, the probability that the sub-Hankel matrix is less than rank $l$ is less than $\gamma^K$.\\
\end{proposition}

The following corollary immediately follows due to the developments in section II.
\begin{corollary}
Let $(\Lambda, U, V)$ be the eigenvalues, left and right eigenvectors underlying the data in the full Hankel matrix. Given any $\beta >0$, and that a sub-Hankel matrix $\hat{H}$ is chosen as in Proposition \ref{RPOD_prop}, the same $(\Lambda, U, V)$ triple can be extracted from the sub-Hankel matrix $\hat{H}$ with probability at least $(1-\beta)^2$, and hence, with probability $(1-\beta)^2$, the information contained in $H$ and $\hat{H}$ is identical in terms of the $(\Lambda, U, V)$ triple.\\
\end{corollary}
\begin{remark}
Several remarks are made below about the above results.
\begin{enumerate}
\item The fractions $\bar{\epsilon}_X$ and $\bar{\epsilon}_Y$  are metrics of the ``difficulty" of the problem. For instance, if all the relevant modes were controllable/observable from every input/output, then these fractions are unity, and any RPOD choice would have rank $l$. The lower these fractions are, the higher the number of rows and columns $sm_2$ and $rm_1$ need to be chosen such that Proposition \ref{RPOD_prop} holds for the sampled sub-Hankel matrix. This corresponds to a mode, or set of modes, being controllable/ observable only from a very sparse set of actuator/ sensor locations respectively. 
\item We do not know $\bar{\epsilon}_X,\bar{\epsilon}_Y$ a priori, and thus, we cannot directly apply Proposition \ref{RPOD_prop}. In practice, we repeatedly sample sub-Hankel matrices, and check the underlying eigenmodes from each choice. If the underlying modes from different choices are identical, then we can give a guarantee that the Hankel matrix is actually rank $l$, given a difficulty level $\bar{\epsilon}$. Thus, we are able to quantify the confidence in our ROMs for different values of the difficulty level $\bar{\epsilon}$.  Typically, we have seen that if the number of rows/ columns sampled are large enough, we are able to extract all the relevant modes.
\item We can also vary the size of the sampled sub-Hankel matrices which in turn raises the probability of sampling a random choice with rank equal to that of the full Hankel matrix. 
\item If we have a priori knowledge of the system, we can sample the sub-Hankel matrix using some sampling distribution other than the uniform distribution function, which as mentioned previously, has the effect of raising the fractions $\bar{\epsilon}_{X},\bar{\epsilon}_Y$, and thus, lower the required size of the sub-Hankel matrix.
\item In reality, the Hankel matrix is not exactly rank $l$ but approximately rank $l$. In such a case, we can appeal to Proposition \ref{P3} to show that the errors incurred due to this fact is small if the contribution from the modes other than the dominant $l$ modes are small.
\end{enumerate}
\end{remark}
\subsection{COMPARISON WITH BALANCED POD} \label{section 4}
We assume that the system has $p$ inputs and $q$ outputs, and suppose that we have $M_1$ snapshots for each input trajectory, and $M_2$ snapshots for each output trajectory.

For Balanced POD, we need to solve the SVD problem of the full Hankel matrix (\ref{hankel}). Here, $H$ is a $qM_2 \times pM_1$ matrix , so Balanced POD has to solve a $(qM_2) \times (pM_1)$ SVD problem which takes time$\varpropto o(max(q^3M_2^3,p^3M_1^3))$. 

Compared with BPOD, the computations required to construct the Hankel matrix as well as that required to solve the SVD problem can be saved by using RPOD. 

We randomly pick $m_1$, $m_2$ snapshots from primal simulation and adjoint simulation respectively, and we randomly choose a set of inputs/outputs matrix $\hat{B}/  \hat{C}$ to construct the reduced order Hankel matrix. Thus, we need to store a $rm_2 \times sm_1$ matrix instead of a $qM_2 \times pM_1$ matrix. From Table \ref{table:compute}, we can see that the computation time using RPOD is $\frac{rs}{pq}\frac{m_1m_2}{M_1M_2}$ time that of using BPOD, for construction of the Hankel matrix, and the computation time for solving the SVD problem using RPOD is $\frac{max(r^3m_2^3,s^3m_1^3)}{max(q^3M_2^3,p^3M_1^3)}$ of using BPOD. 

In Table \ref{table:compute}, we show the comparison of the computational requirements of RPOD and BPOD.  Note that the order of the system N can be very large for most realistic problems. 

\begin{table}[htbp]
\caption{Computational analysis for RPOD and BPOD}
\centering
\begin{tabular}{c|c|c}
\hline
& Compute Markov parameters & SVD size \\
\hline
BPOD& $(M_1+M_2)pqN^2$&$(qM_2 \times pM_1)$ \\
\hline
RPOD &$(m_1+m_2) rsN^2$&$(rm_2 \times sm_1)$\\
\hline
\end{tabular}
\label{table:compute}
\end{table}


\section{COMPUTATIONAL RESULTS}\label{section 5}
In the following, we will show the comparison of RPOD with Balanced POD for three examples: a pollutant transport equation, the linearized channel flow problem and probability density function evolution in a 2-D damped duffing oscillator.

We define the output error as:
\begin{eqnarray}
E_{output} =\frac{ \|Y_{true} - Y_{red} \|}{\|Y_{true}\|},
\end{eqnarray}
where $Y_{true}$ are the outputs of the true system and $Y_{red}$ are the outputs of the reduced order system. 

The state error is defined as:
\begin{eqnarray}
E_{state}= \frac{\| X_{true} - X_{red} \|}{\|X_{true}\|},
\end{eqnarray}
where $X_{true}$ is the state of the true system and $X_{red}$ is the state of the reduced order system. 

\subsection{Pollutant transport equation}\label{pollutant}
The two-dimensional advection-diffusion equation describing the contaminant transport is:
\begin{eqnarray}
\frac{\partial c(x,y,t)}{\partial t} = D_x \frac{\partial^2 c(x,t)}{\partial x^2}+ D_y \frac{\partial^2 c(x,t)}{\partial y^2}-\nonumber \\
v_x \frac{\partial c(x,t)}{\partial x}
- v_y \frac{\partial c(x,t)}{\partial y}+ S_s,
\end{eqnarray}

where $c$ is concentration of the contaminant, $D$ is dispersion and takes value 0.6 here, $v$ is velocity in the $x$ and $y$ directions, and takes value 1, and $S_s$ is source of pollutant. 
In simulation, there are three obstacles and three sources in the field. The initial condition for the simulation is zero. We use Neumann boundary conditions. The model is simulated for a period of $10$ minutes, and the field is a square with the length of each edge $5m$. The field is discretized into a $ 50 * 50 $ grid. The actual field is shown in Fig.\ref{rand}.

\begin{figure}[!htbp]
\centering
\includegraphics[width=2.5in]{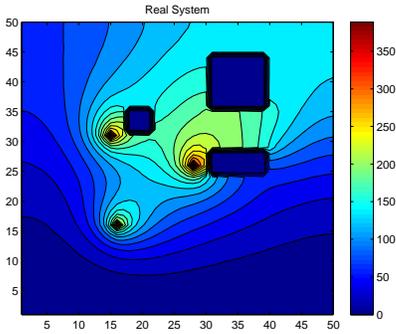}
\caption{Actual field at the end of simulation for pollutant transport equation}
\label{rand}
\end{figure}
The size of system matrix is $2500 \times 2500$, here, we take the impluse response of the system. For BPOD, we use the full state measurements. Since there are only 3 sources in the field, thus, we need to take enough snapshots for the primal simulation to make sure the number of active modes should be less than or equal to the number of input trajectories. Here we use 500 snapshots from $t \in [0min, 10min]$ for the primal simulation. Similarly, Assumption \ref{A1} is guaranteed by using the full state measurements for the adjoint simulation, thus, we use 3 snapshots from $t \in [0min, 1min]$ for the adjoint simulation, which will not result in a large SVD problem. Notice that taking the snapshots earlier will allow us to extract more modes before they die out. Therefore, for BPOD, we need to solve a $7500 \times 1500$ SVD problem. For RPOD, we randomly choose 500 measurements, and take 300 snapshots from $t \in [0min, 10min]$ for the primal simulation, and 3 snapshots from $t \in [0min, 1min]$ for the adjoint simulation. Thus, we only need to solve a $1500 \times 900$ SVD problem for RPOD. We extract 80 modes using both methods, and construct the ROM using these modes. In Fig.\ref{fig:pollerrors}(a), we show the comparison of the first twenty eigenvalues extracted by two methods.

\begin{figure}[!htbp]
\centering

\subfigure[Comparison of eigenvalues extract by RPOD and BPOD for pollutant transport equation]{
\includegraphics[width=2.4in]{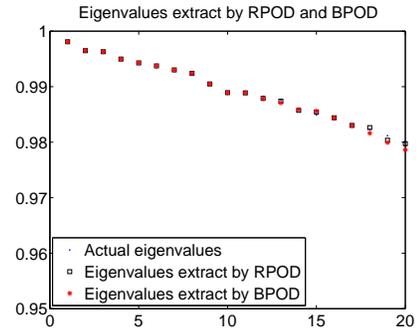}}
\subfigure[Comparison of output errors]{\includegraphics[width=2.5in]{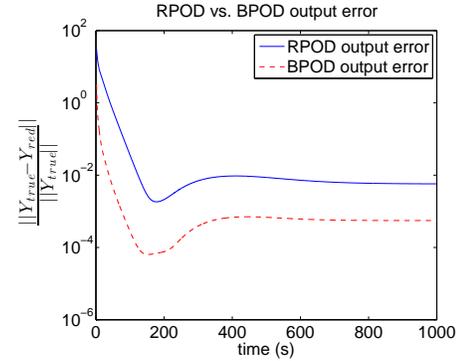} }
\caption{Comparison between RPOD and BPOD for Pollutant Transport Equation}
\label{fig:pollerrors}
\end{figure}

To test the ROM, we take the average output/state error over the 3 different impluse responses. The state errors and the output errors are the same because we take the full state measurements, thus, we show the comparison of the output errors in Fig. \ref{fig:pollerrors}(b). We can see that BPOD is more accurate than RPOD, but both the errors are less than $1\%$, however, there is significant computational savings in using the RPOD over the BPOD in solving the SVD problem.

\subsection{ Linearized Channel Flow problem}\label{channel}
Consider the problem of the fluid flow in a plane channel. We focus on the linearized case when there are small perturbations about a steady laminar flow. The flow is perturbed by body force $B(y,z) f(t)$, which means the force is acting in the wall-normal direction. There is no-slip boundary condition at the walls $y= \pm 1$ and the flow is assumed to be periodic in the $x$ and $z$ direction. Assume there is no variations in the $x$ direction, then the linearized equation of the wall-normal velocity $v$ and the wall-normal vorticity $\eta$ are given by:
\begin{eqnarray}
\frac{\partial v}{\partial t} = \frac{1}{R}  \nabla^2 v + B f,\nonumber \\
\frac{\partial \eta}{\partial t} = \frac{1}{R} \nabla^2 \eta - U' \frac{\partial v}{\partial z},
\end{eqnarray}
where $R = 100 $ is the Reynolds number and $U(y)= 1-y^2$ is the steady state velocity. The domain $z \in [0, 2\pi]$. We discretize the system using the finite difference method, where both the $y$ direction and $z$ direction are discretized into 21 nodes. Thus, the size of the system is $882 \times 882$.
There are 2 constant body forces on $y = 0$, and the measurements are taken on all the nodes on boundaries. For BPOD, we use 80 measurements on the boundaries, and take 1000 snapshots from $t \in [0,1000s]$ for the primal simulation, 50 snapshots from $t \in [0, 500s]$ for the adjoint simulation, which leads to a $8000 \times 2000$ SVD problem. For RPOD, we randomly choose 50 measurements from the 80 measurements on the boundaries, and take 200 snapshots from $t \in [0, 200s]$ for the primal simulation, and take 20 snapshots from $t \in [0, 200s]$ for the adjoint simulation. Thus, we need to solve a $2000 \times 400$ SVD problem for RPOD. The actual velocity and vorticity at $t=1000s$ are shown in Fig. \ref{cfield}.

\begin{figure}[!htbp]
\centering
\subfigure[Actual velocity at t=1000s]{\includegraphics[width=2.4in]{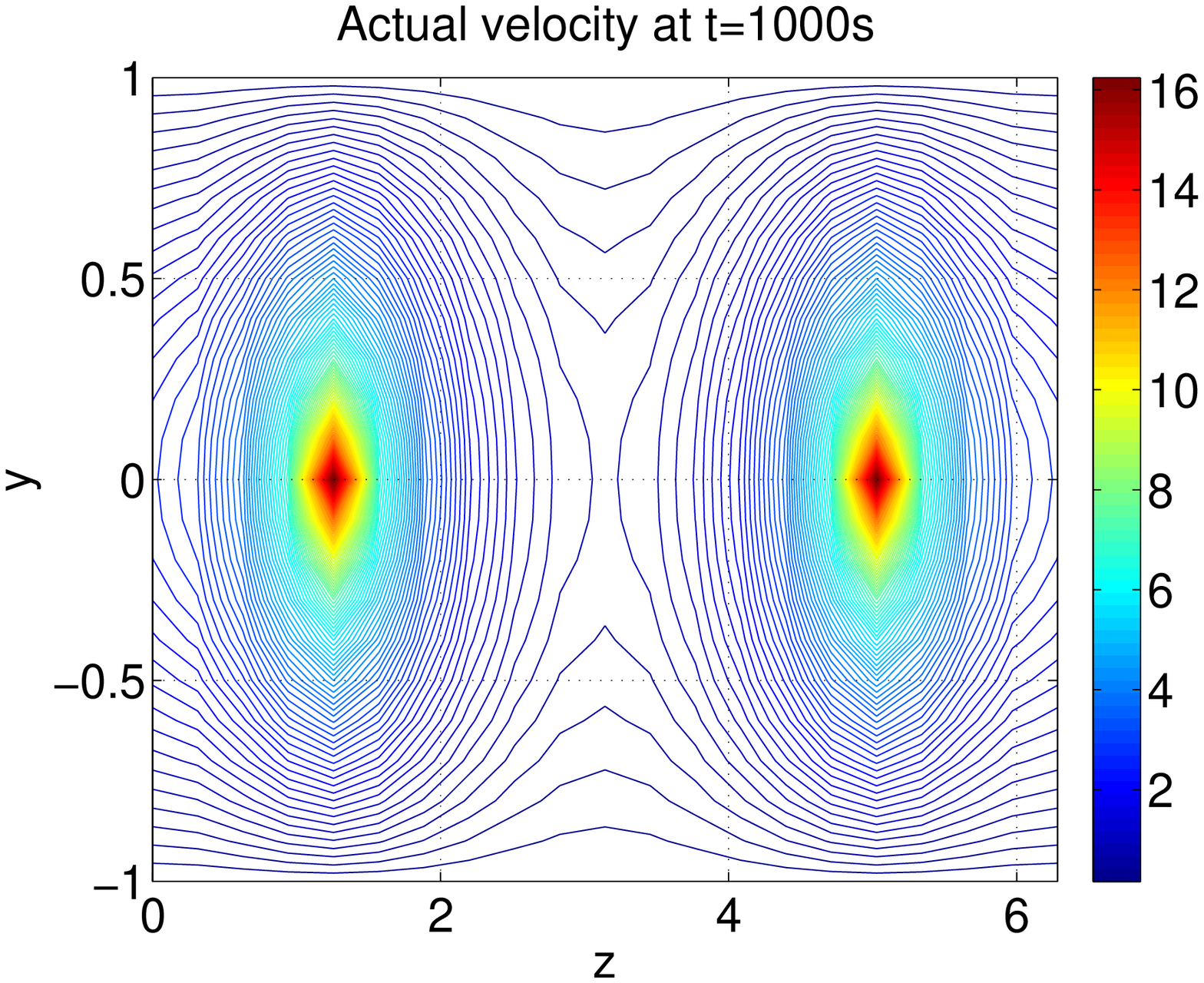} }
\subfigure[Actual vorticity at t=1000s]{
\includegraphics[width=2.4in]{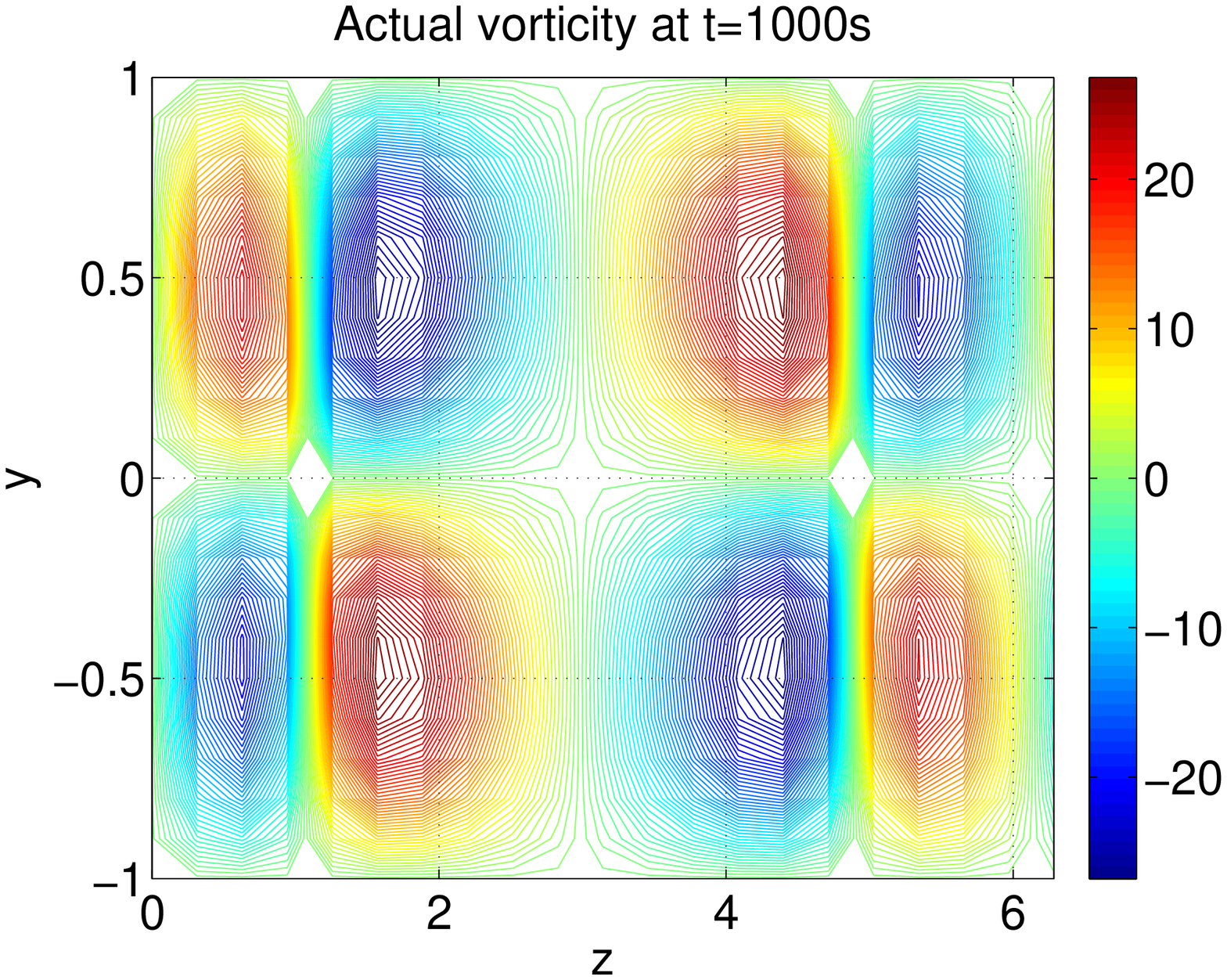}}
\caption{Actual velocity and vorticity of the channel flow problem}
\label{cfield}
\end{figure}

In Fig. \ref{channel_velocity}, we compare the velocity modes of the system using RPOD with the actual velocity modes. Also, we compare the first three vorticity modes of the system using RPOD with the actual vorticity modes in Fig. \ref{channel_vorticity}.

\begin{figure}[!htbp]
\centering
\subfigure[Actual first velocity mode]{
\includegraphics[width=1.66in]{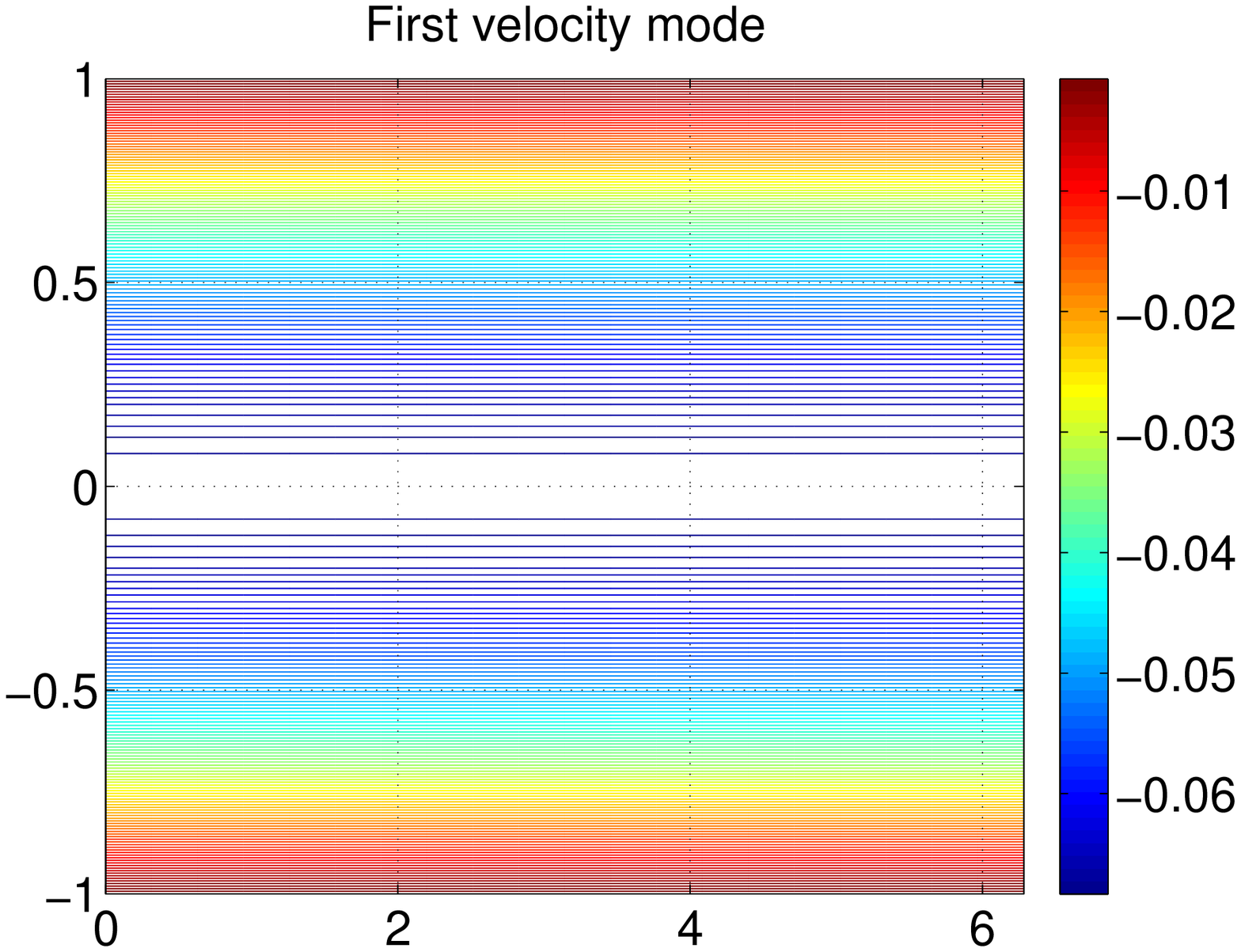}}
\subfigure[ROM first velocity mode]{
\includegraphics[width=1.66in]{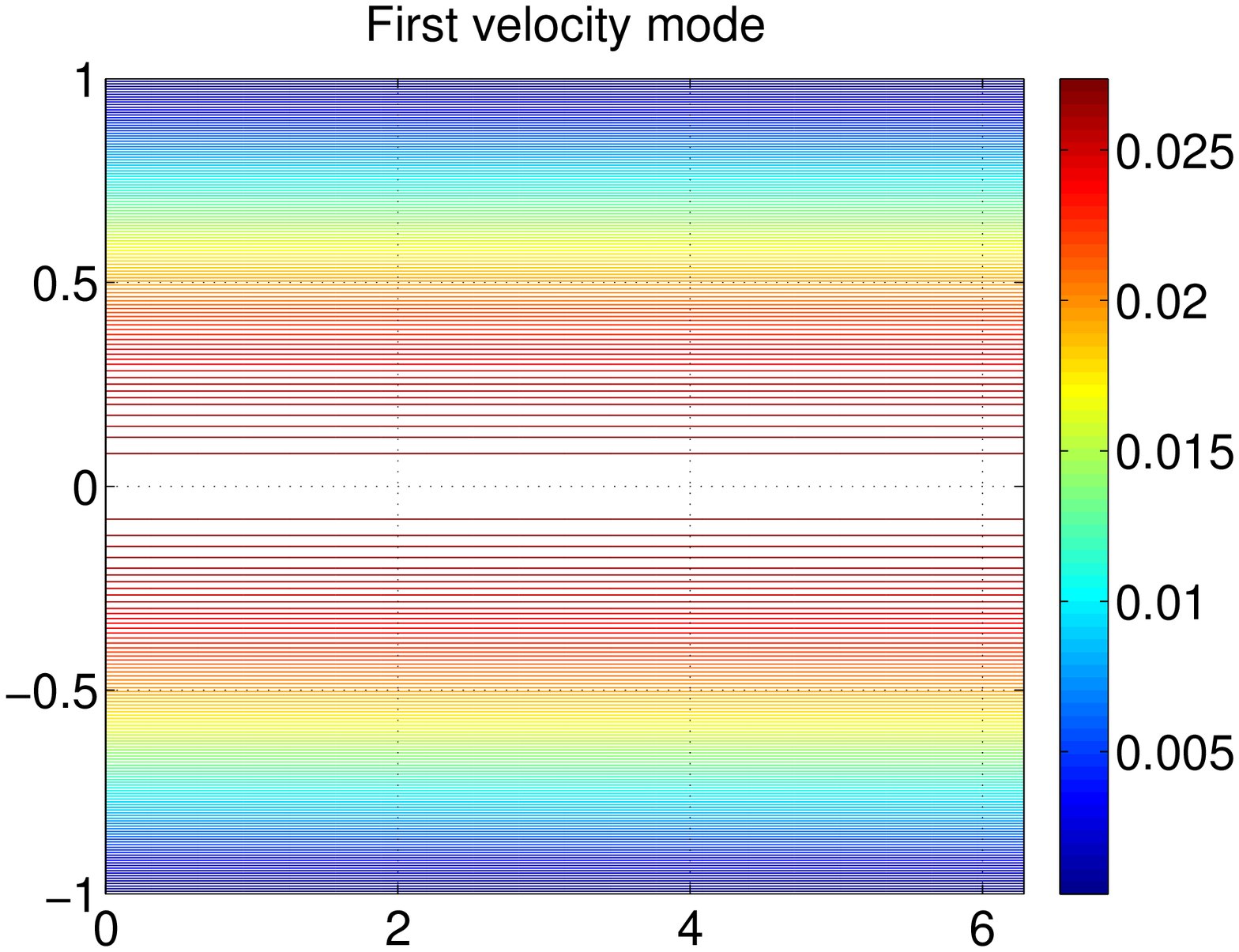}}
\subfigure[Actual second velocity mode]{\includegraphics[width=1.66in]{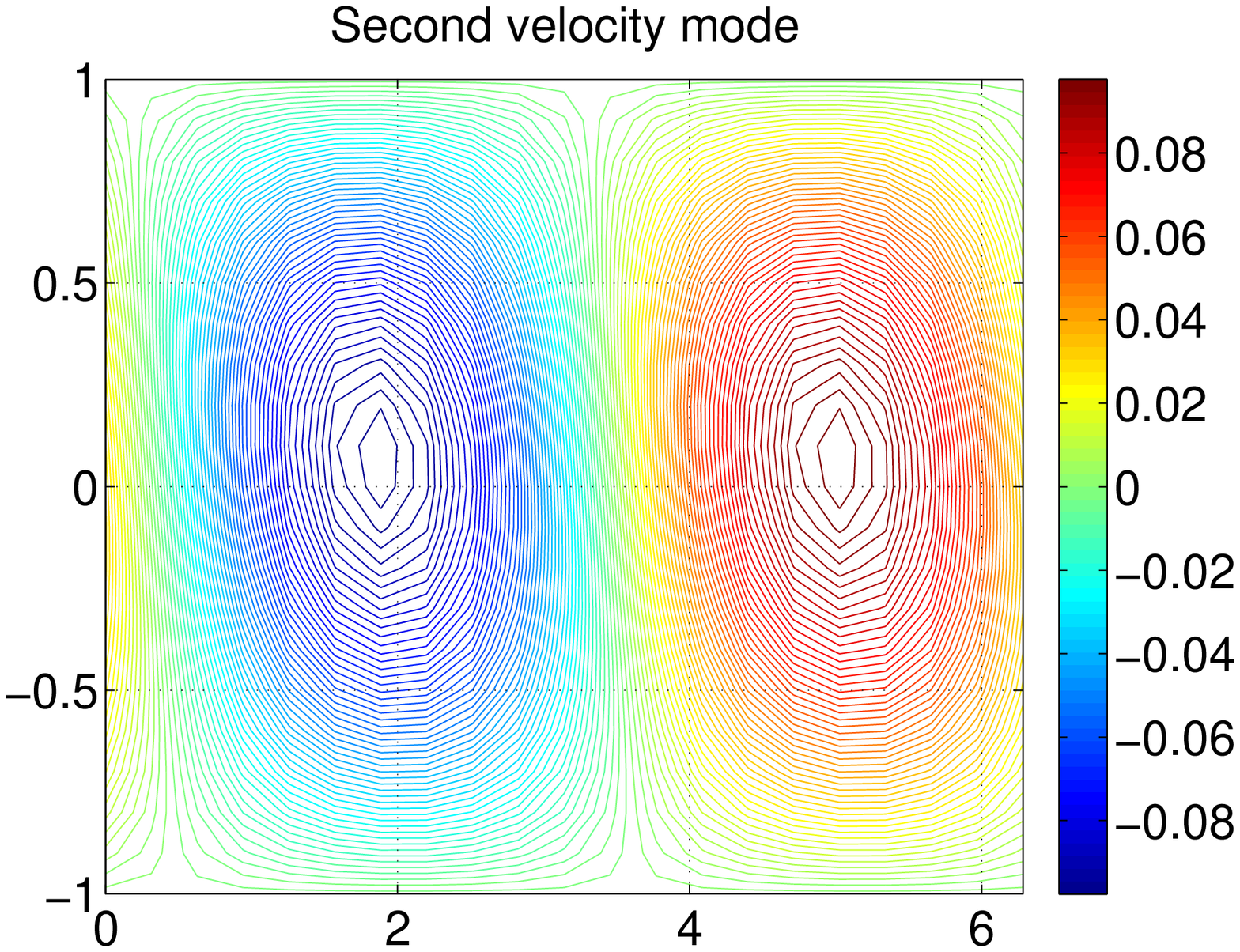}}
\subfigure[ROM second velocity mode]{\includegraphics[width=1.66in]{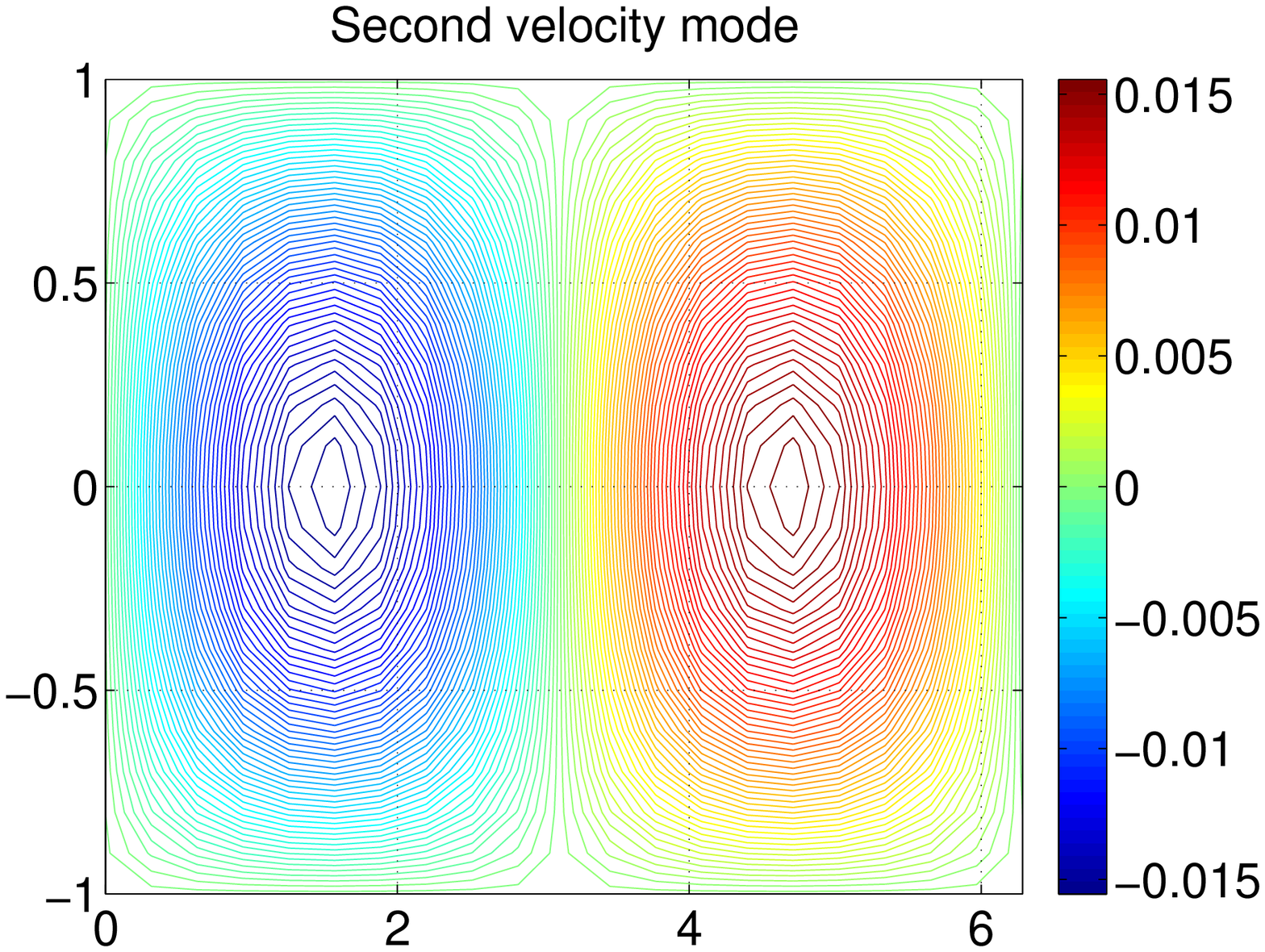}}
\subfigure[Actual third velocity mode]{\includegraphics[width=1.66in]{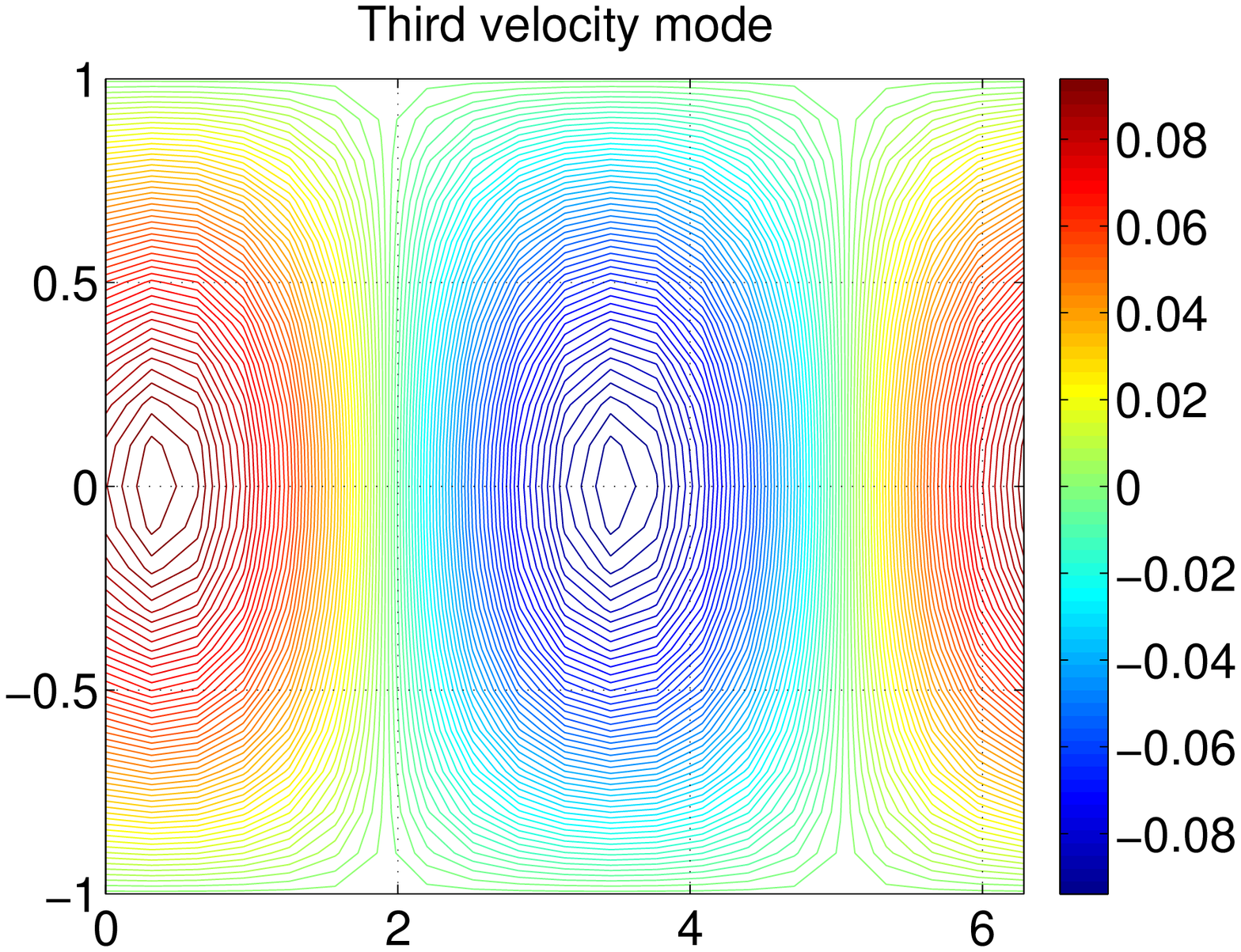}}
\subfigure[ROM third velocity mode]{\includegraphics[width=1.66in]{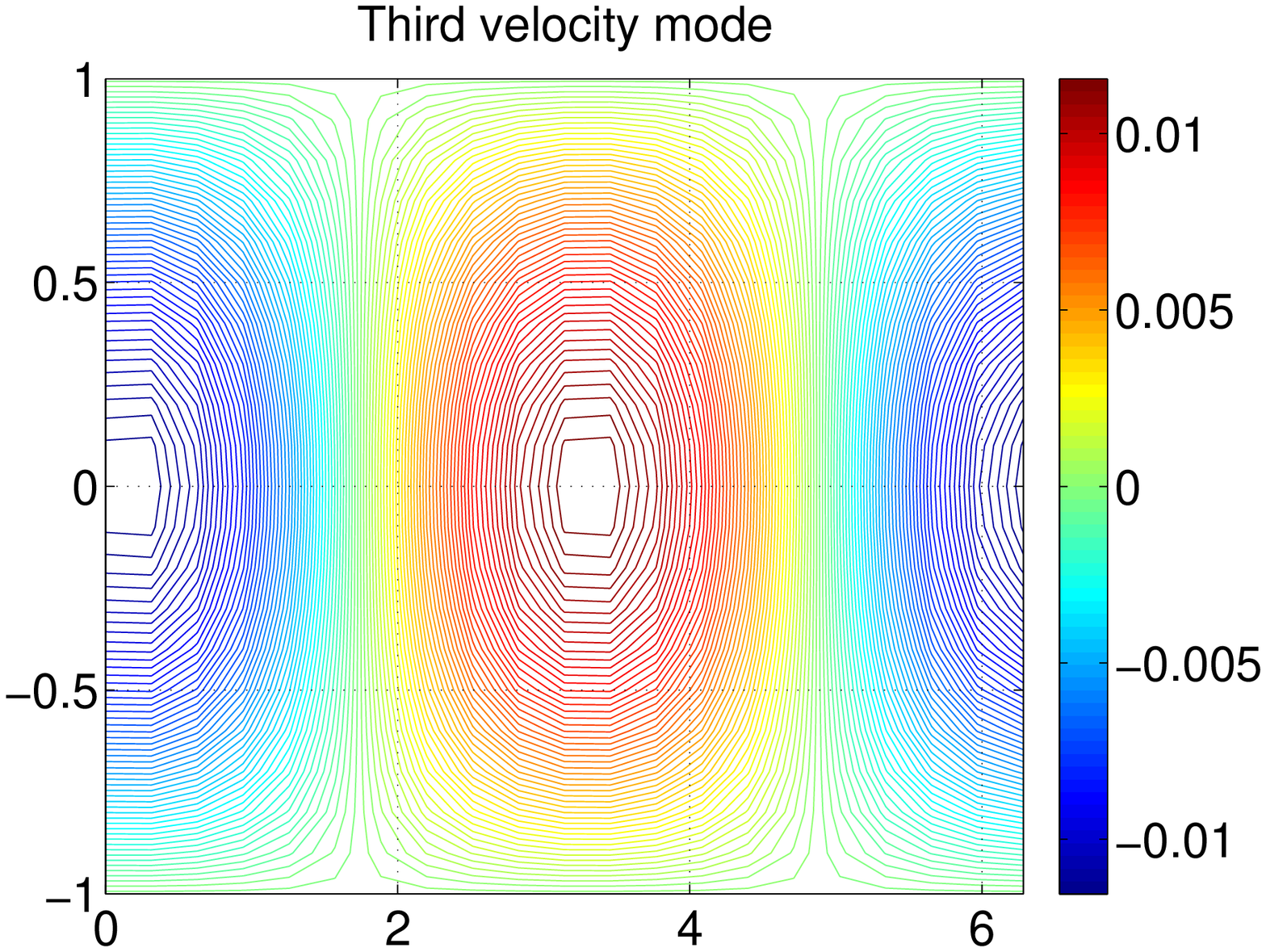}}
\caption{Comparison between ROM and actual velocity modes}
\label{channel_velocity}
\end{figure}

\begin{figure}[!htbp]
\centering
\subfigure[Actual first vorticity mode]{
\includegraphics[width=1.66in]{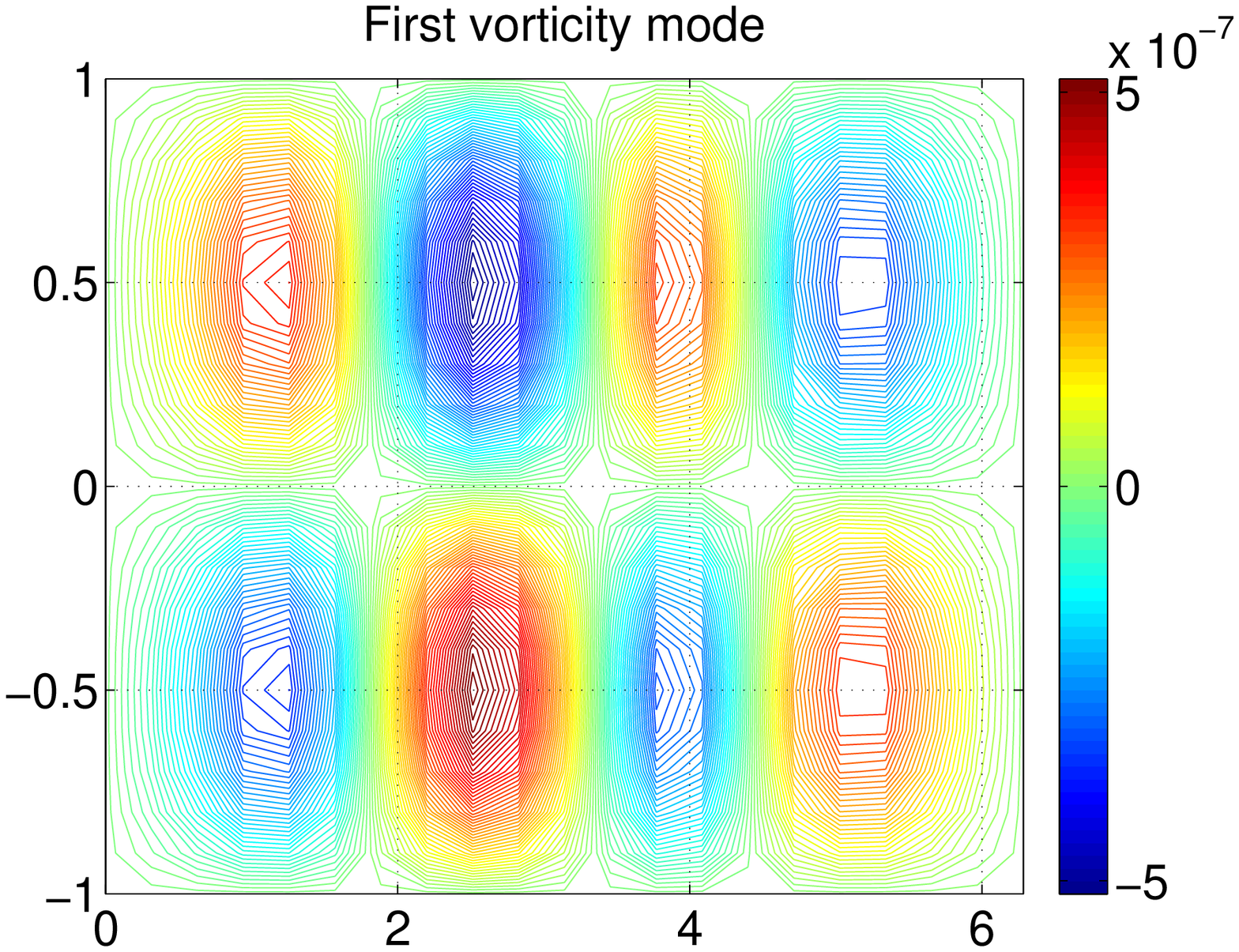}}
\subfigure[ROM first  vorticity  mode]{
\includegraphics[width=1.66in]{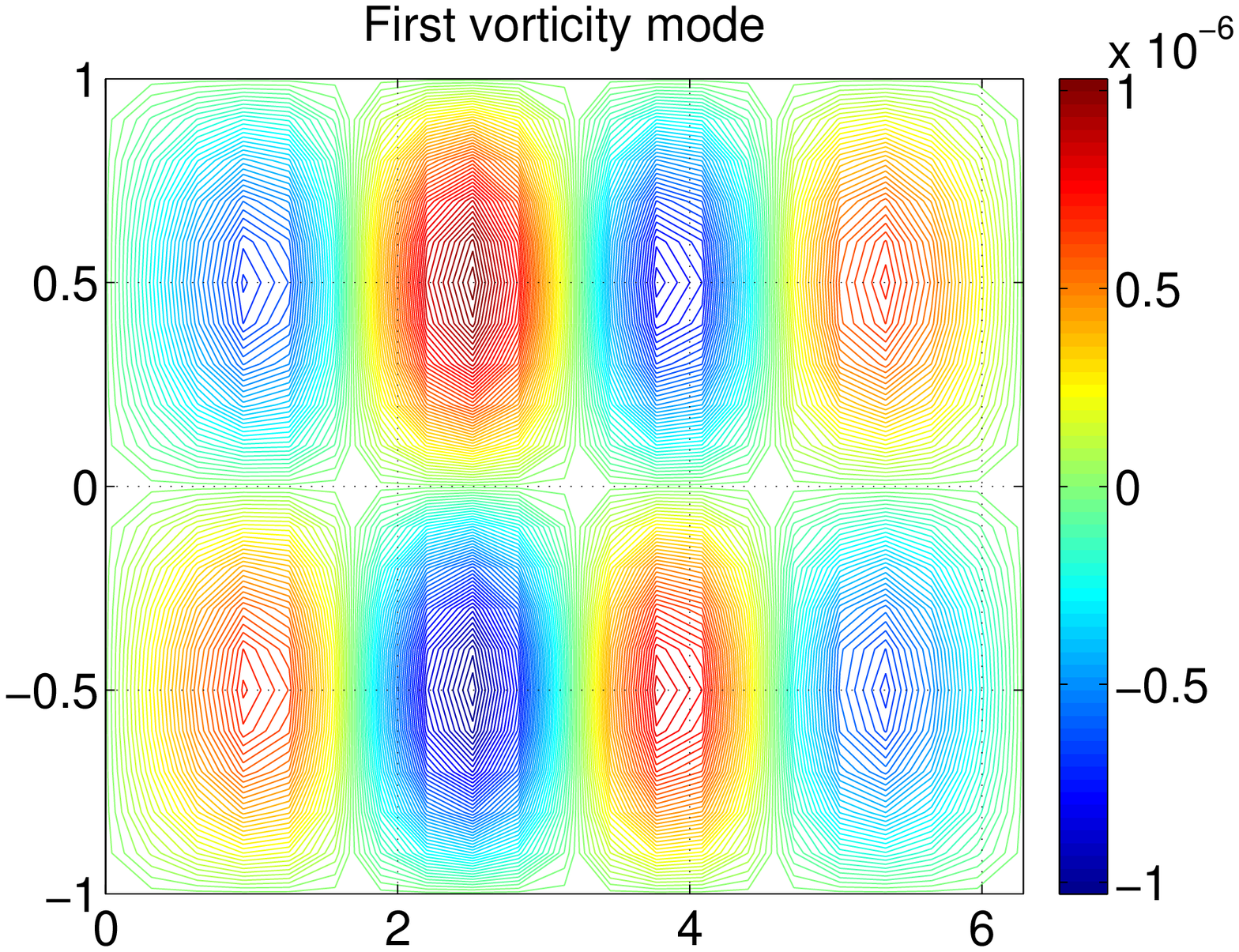}}
\subfigure[Actual second  vorticity  mode]{\includegraphics[width=1.66in]{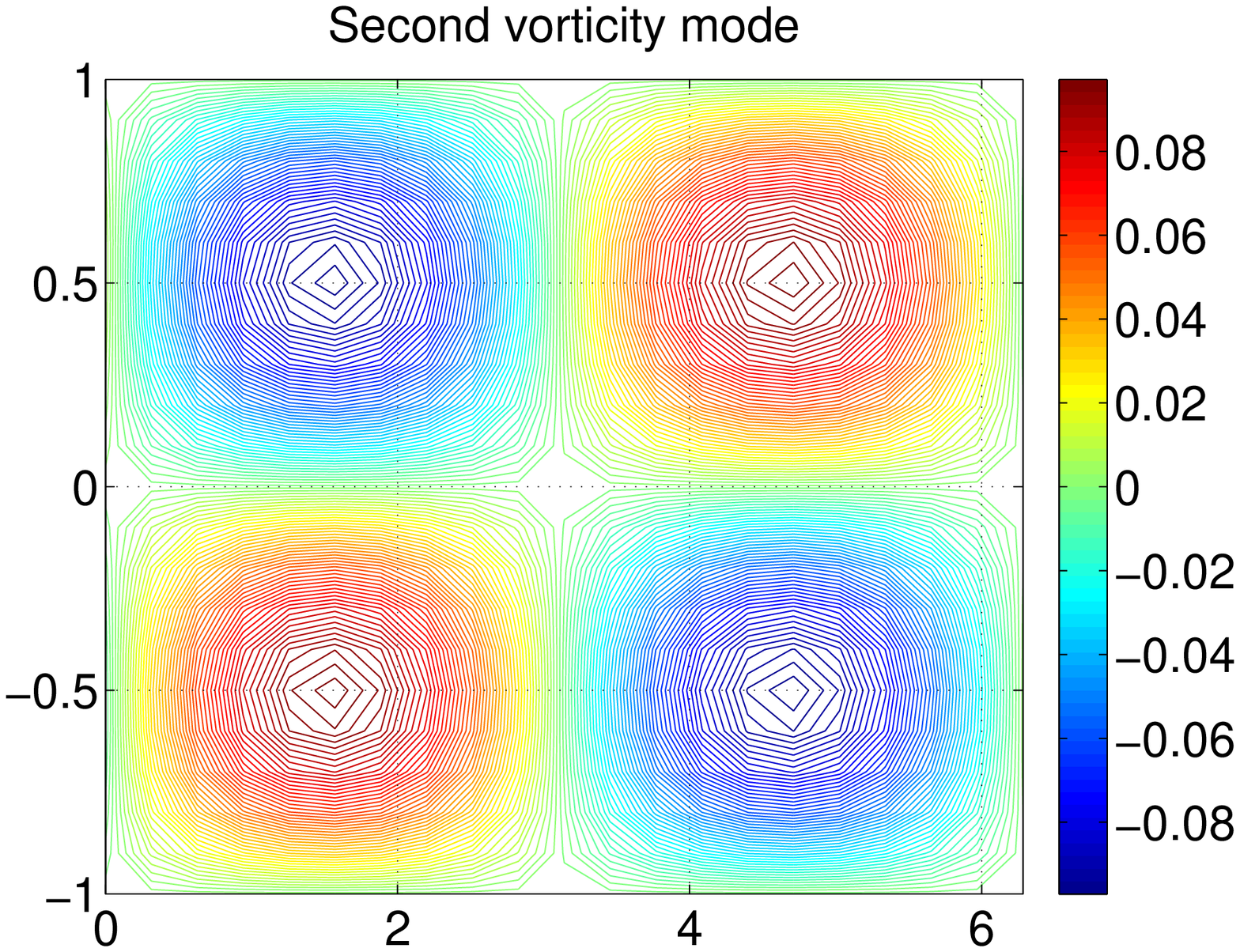}}
\subfigure[ROM second  vorticity  mode]{\includegraphics[width=1.66in]{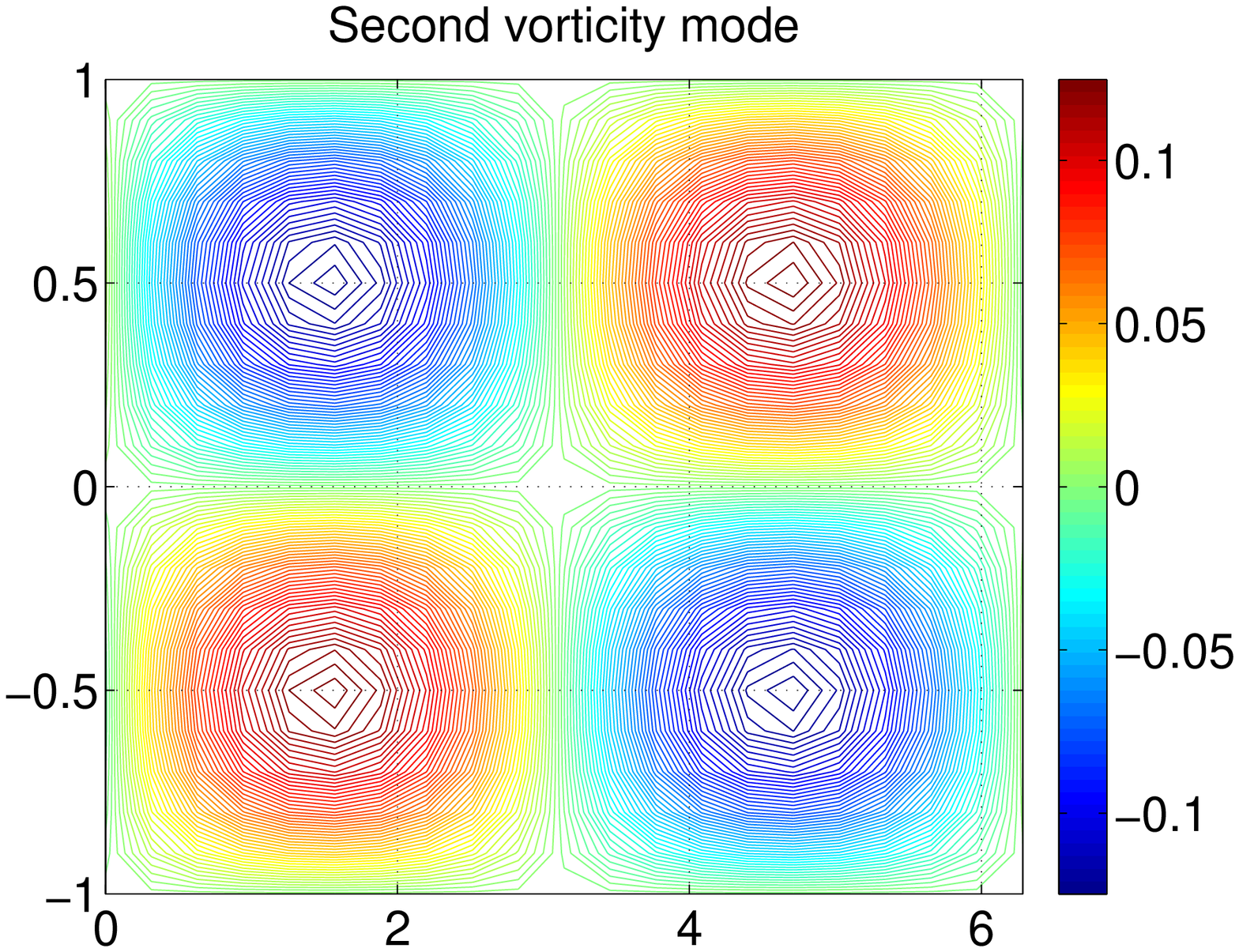}}
\subfigure[Actual third  vorticity  mode]{\includegraphics[width=1.66in]{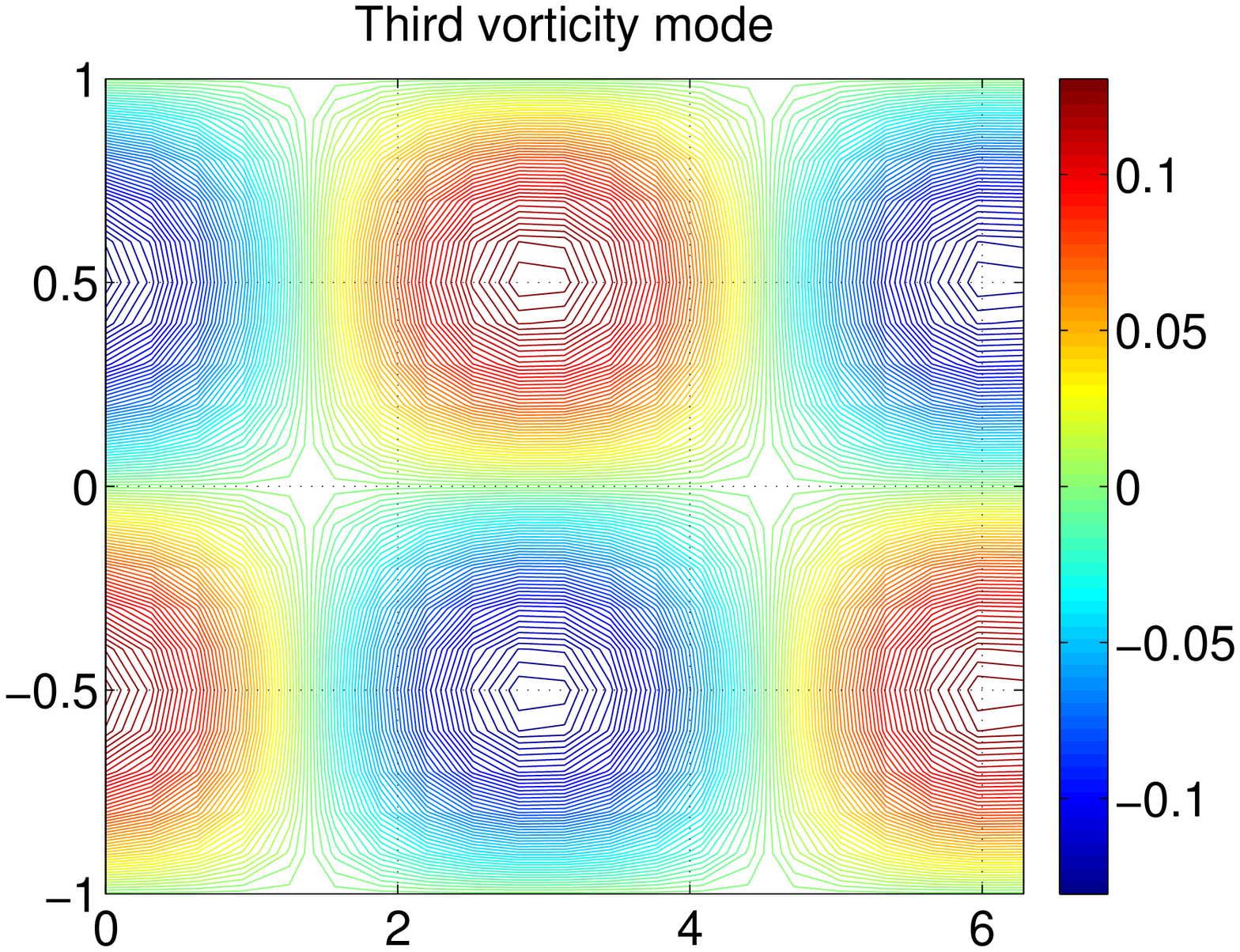}}
\subfigure[ROM third  vorticity  mode]{\includegraphics[width=1.66in]{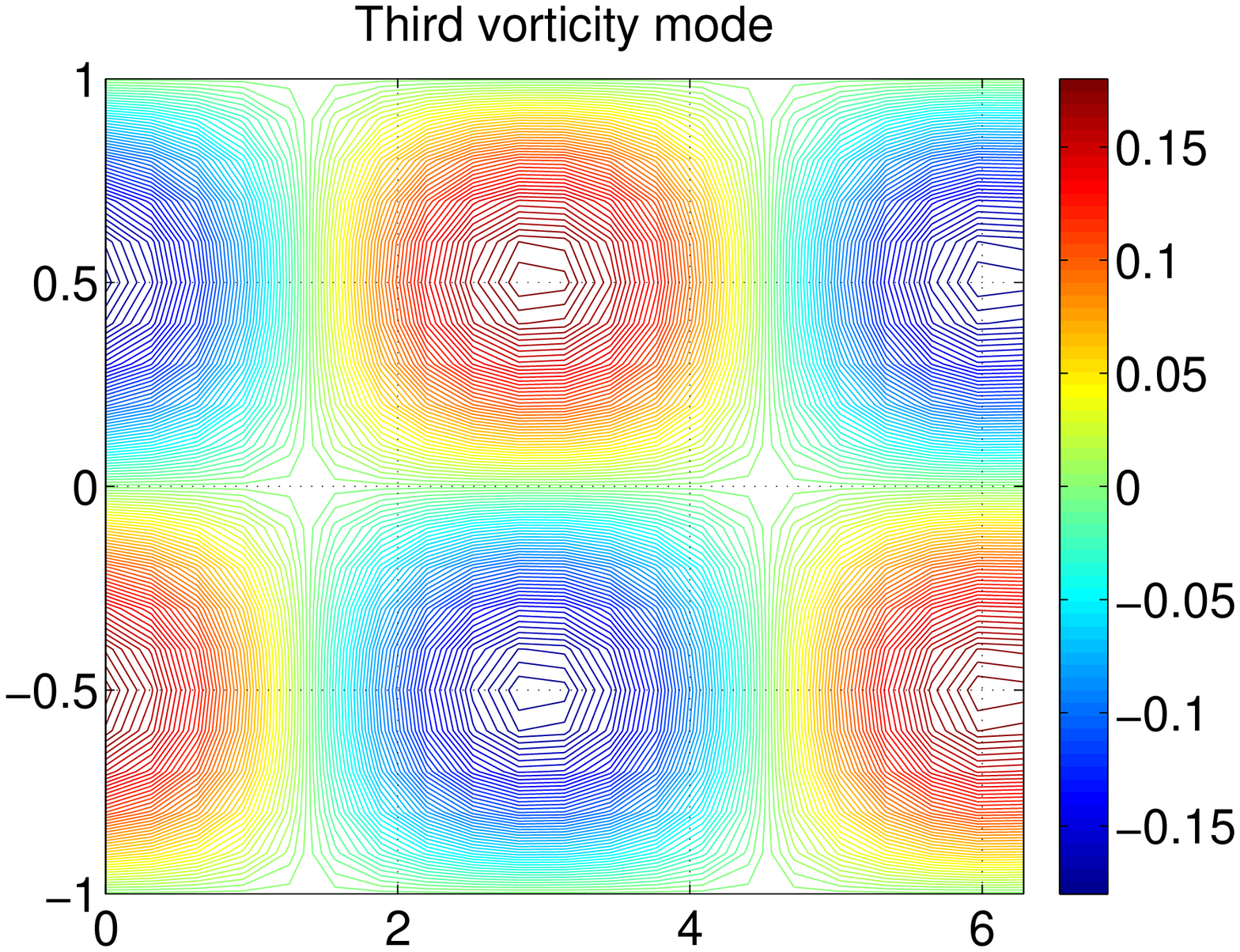}}
\caption{Comparison between ROM and actual vorticity modes}
\label{channel_vorticity}
\end{figure}

Here, we should note that the sign and the modulus of the ROM velocity modes or vorticity modes are not the same as the actual modes, however, if needed, we can rescale the ROM modes to make them match.  For both methods, we extract 40 modes, the first 30 extracted eigenvalues are compared in Fig. \ref{fig:channeleigenvalues}.
\begin{figure}[!htbp]
\centering
\includegraphics[width=2.5in]{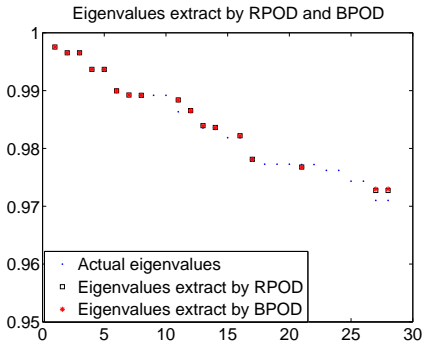}
\caption{Comparison of eigenvalues extract by RPOD and BPOD for  channel flow problem}
\label{fig:channeleigenvalues}
\end{figure}
 
The comparison of the state errors and output errors are shown in Fig. \ref{fig:cerrors}. To test the ROM, we use 20 different white noise forcings and take the average output/state error over these 20 simulation. We can see that the eigenvalues extracted by RPOD and BPOD are almost the same. In this simulation, we notice that at first, the state error and output error using BPOD are slightly better than using RPOD, but after some time, the errors are almost the same. The output errors using both methods are less than $ 0.1\%$, and the state errors using both methods are around $5\%$. Thus, we can conclude that RPOD is comparable to BPOD but requires far less computation.

\begin{figure}[!htbp]
\centering
\subfigure[Comparison of output errors]{
\includegraphics[width=2.5in]{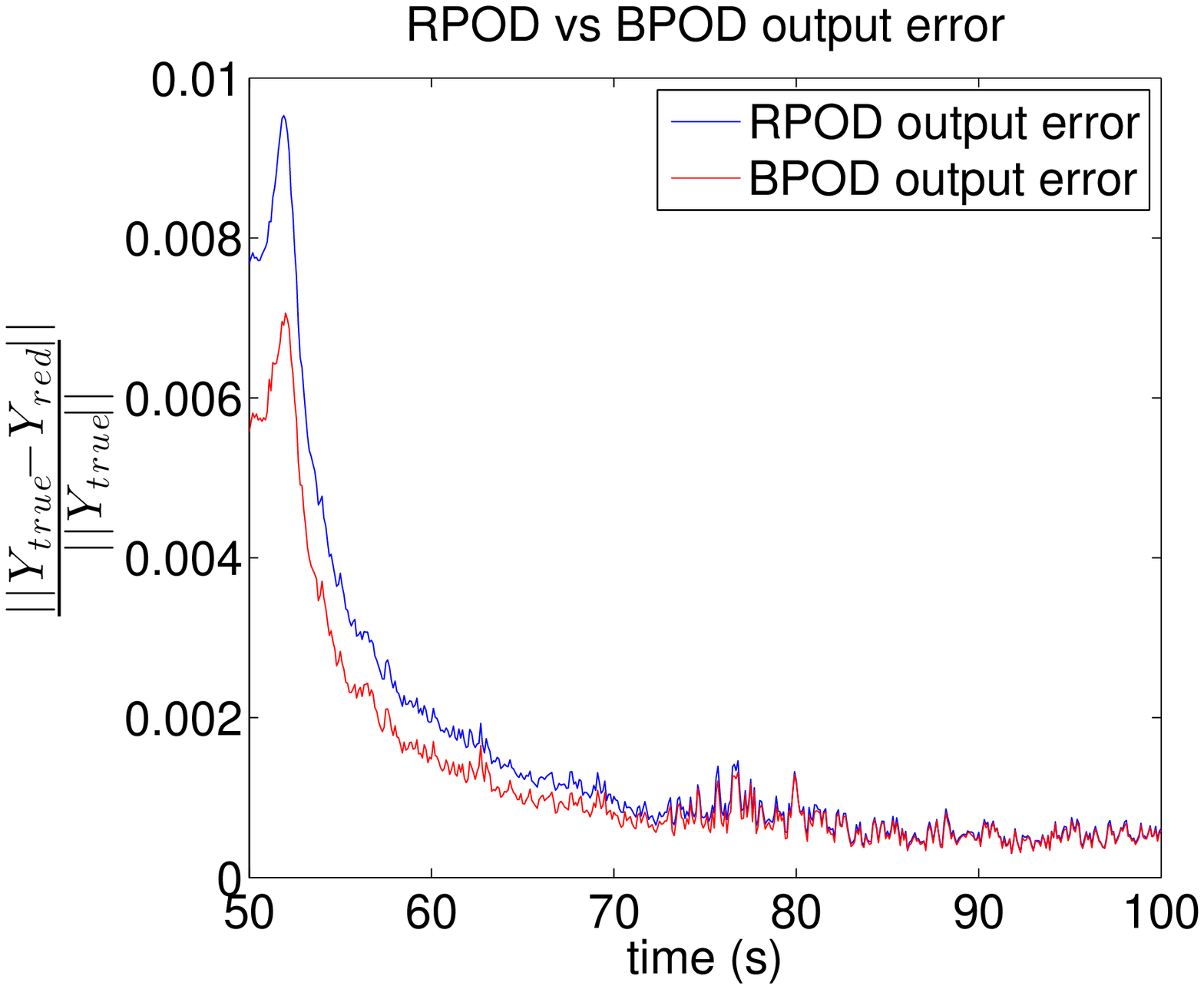}}
\subfigure[Comparison of state errors]{\includegraphics[width=2.5in]{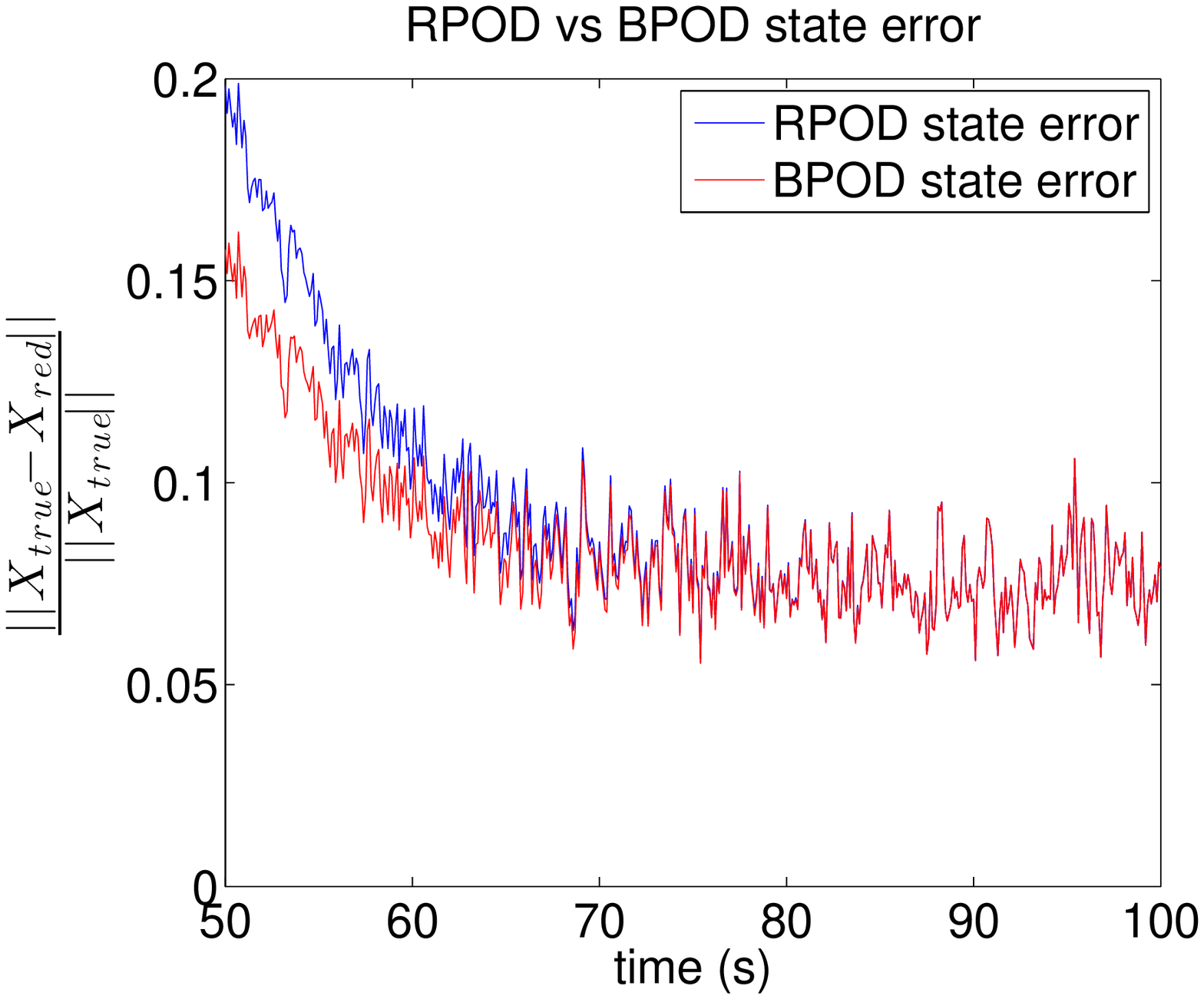}}
\caption{Comparison between RPOD and BPOD for channel flow problem}
\label{fig:cerrors}
\end{figure}

\subsection{Probability density function evolution in a 2-D damped Duffing oscillator}
The 2-D damped Duffing oscillator is: 
\begin{eqnarray}
\ddot{x}+ \eta\dot{x} + \alpha x + \beta x^3= g G(t)
\end{eqnarray}
Here, $\eta = 10$, $ \alpha = -15$, $\beta = 30$, $g=1$ (soft-spring case). 

If we are to propagate a probability density function through this system, it leads to the Fokker-Planck-Kolmogorov Equation. 
\begin{eqnarray}
\frac{\partial W(t, x)}{\partial t} =L_{FP} W(t,x),
\end{eqnarray}
where $W(t,x)$ is the probability density of the state, $L_{FP}$ is the Fokker-Planck-Kolmogorov operator, and:
\begin{eqnarray}
L_{FP} = [- \sum_{i=1}^N \frac{\partial}{ \partial x_i} D_i^{(1)}(. , .) + \sum_{i,j=1}^{N} \frac{ \partial ^2}{ \partial x_i \partial x_j } D_{ij}^{(2)} (. , .)]
\end{eqnarray}
\begin{eqnarray}
D^{(1)}(t, x) = f(t, x) + \frac{1}{2} \frac{ \partial g(t, x) } {\partial x} Q g(t, x)\\
D^{(2)}(t, x) = \frac {1}{2} g(t, x) Q g(t, x)^{T}
\end{eqnarray}

The FPK Equation is a linear, parabolic partial differential equation. Using the finite element methods, we discretize the FPK equation into 1176 grids, and we use the RPOD and BPOD method to construct a reduced order model of the FPK equation. 

First, we compare the transient probability density function using RPOD with the full order system in Fig.\ref{fig:trainsent}. We can see that at beginning, the behavior of reduced order model is not good enough, because we don't have enough modes to capture the initial transient behavior. But later, the behavior of ROM is approximately the same as the full order system. 

\begin{figure}[!htbp]
\centering
\subfigure[Actual transient pdf at t=0.2s]{
\includegraphics[width=1.66in]{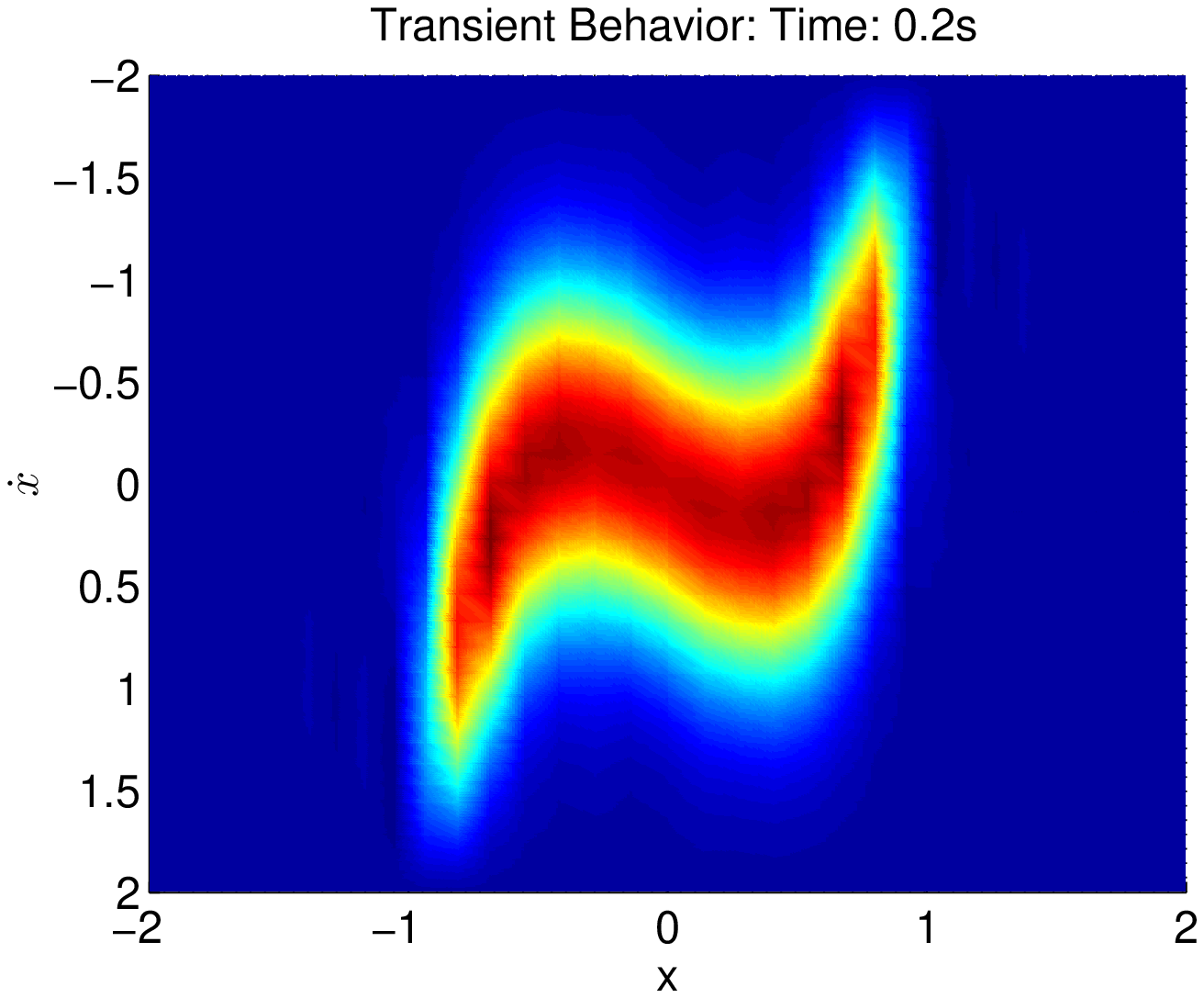}}
\subfigure[ROM transient pdf at t=0.2s]{
\includegraphics[width=1.66in]{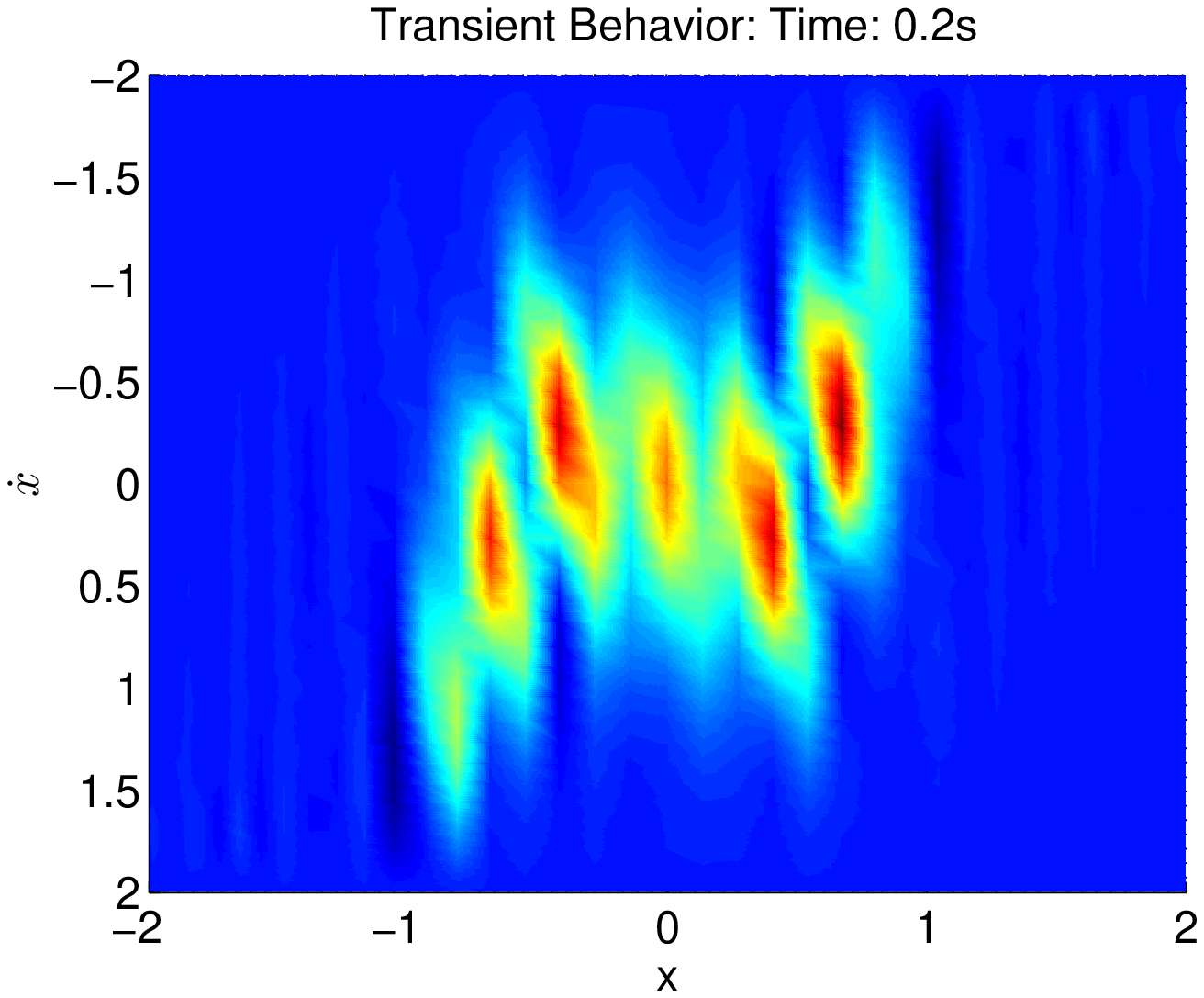}}
\subfigure[Actual transient pdf at t=0.3s]{\includegraphics[width=1.66in]{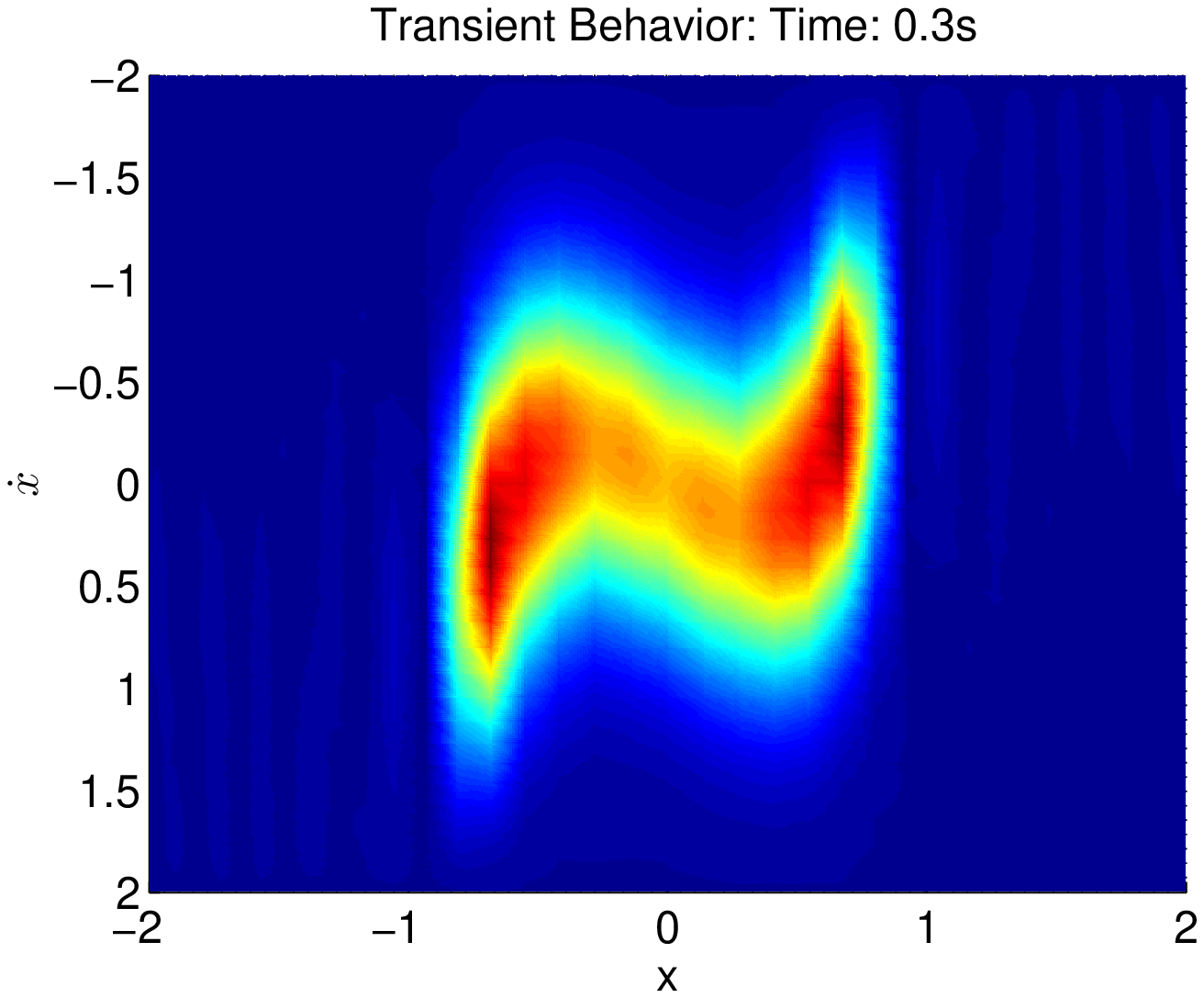}}
\subfigure[ROM transient pdf at t=0.3s]{\includegraphics[width=1.66in]{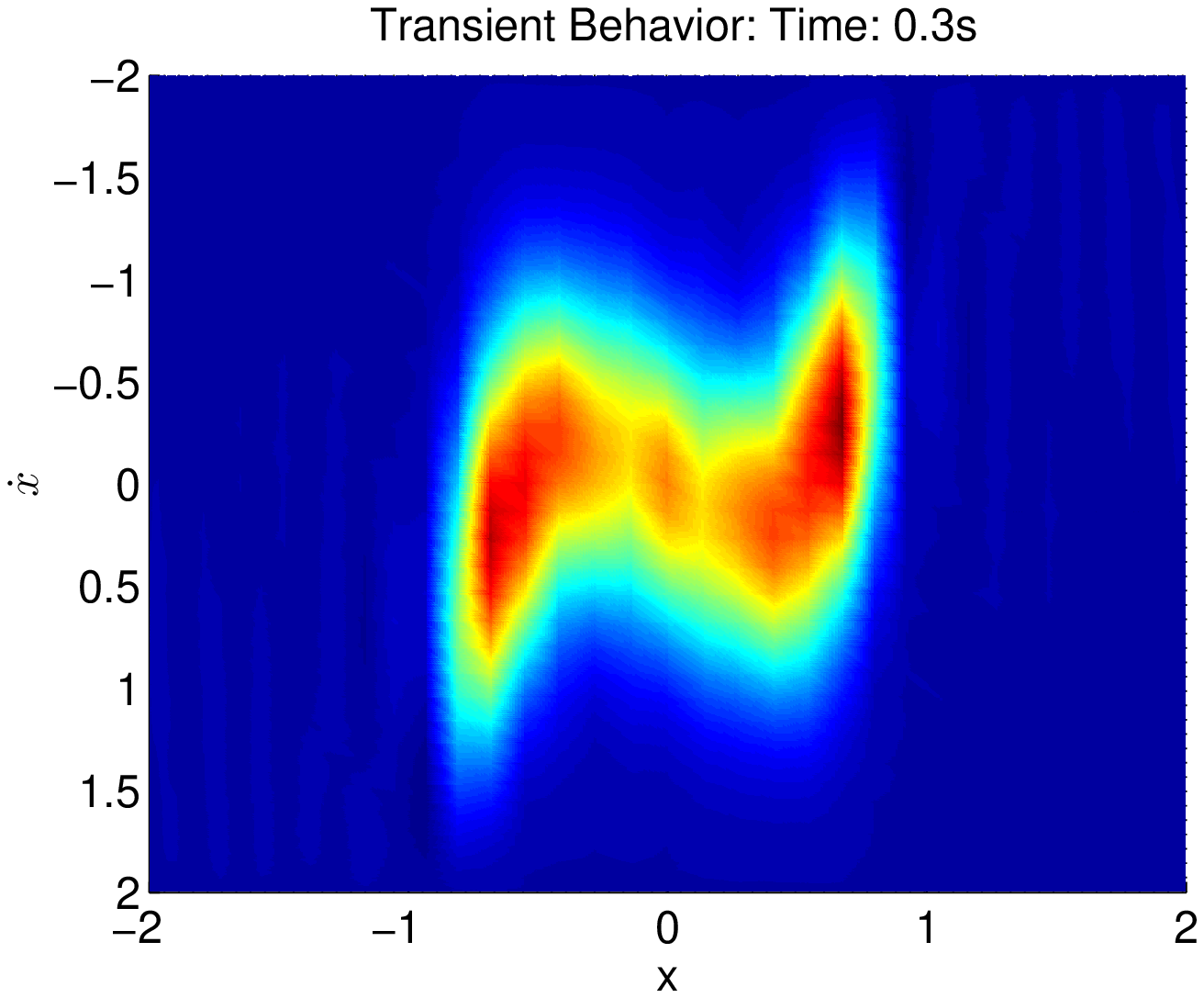}}
\subfigure[Actual transient pdf at t=0.5s]{\includegraphics[width=1.66in]{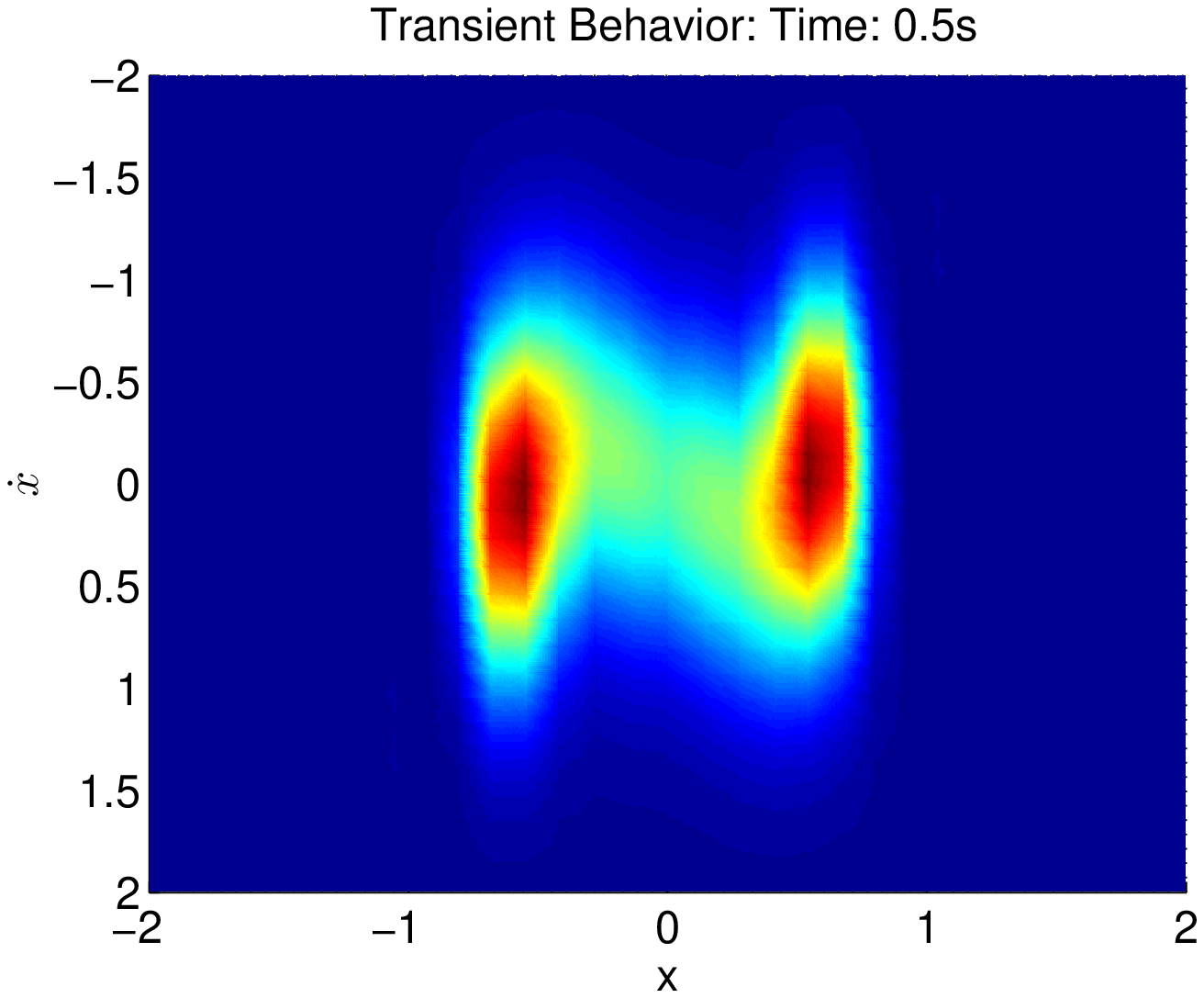}}
\subfigure[ROM transient pdf at t=0.5s]{\includegraphics[width=1.66in]{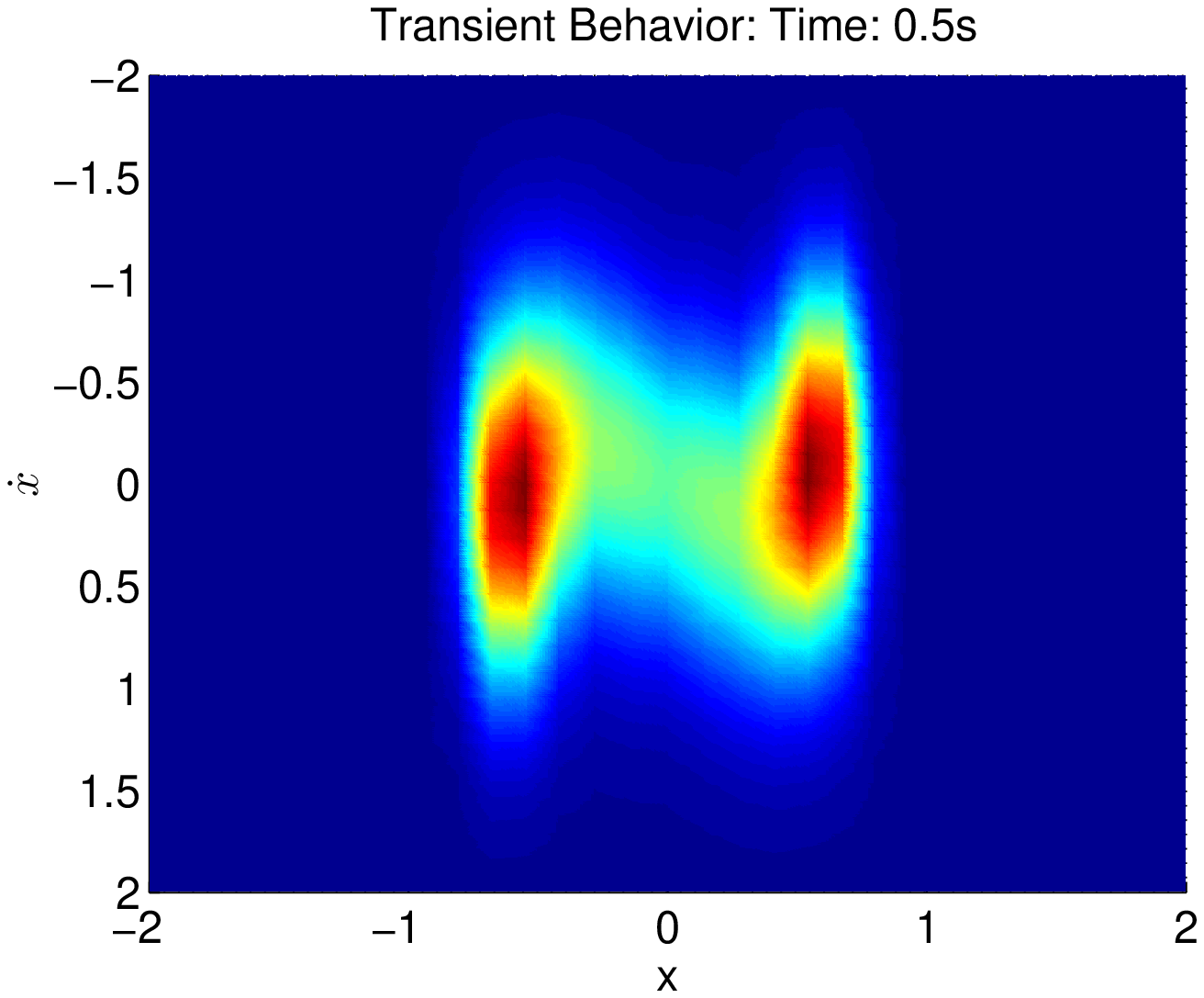}}
\caption{Comparison between ROM and actual transient pdf}
\label{fig:trainsent}
\end{figure}

Next we compare the extracted eigenvalues and state error using RPOD and BPOD. For the FPK equation, we don't have input/ output matrices, so we use different initial conditions for the primal/adjoint simulations of the discretized FPK operator. For BPOD, we take 500 input/output trajectories, and 3 snapshots from $t \in [0.1s, 0.2s]$, which leads to a $(1500 \times 1500)$ SVD problem. For RPOD, we randomly choose 294 input/output trajectories from BPOD, and take 1 snapshot at $t= 0.1s$, so we only need to solve a $(294 \times 294)$ SVD problem. A total of 30 modes is extracted by both methods, and the eigenvalues are compared in Fig. \ref{Fokker}(a), while the state errors are compared in Fig. \ref{Fokker}(b). 

\begin{figure}[!htbp]
\centering
\subfigure[Comparison of eigenvalues extract by RPOD and BPOD for 2D damped duffing oscillator]{
\includegraphics[width=2.4in]{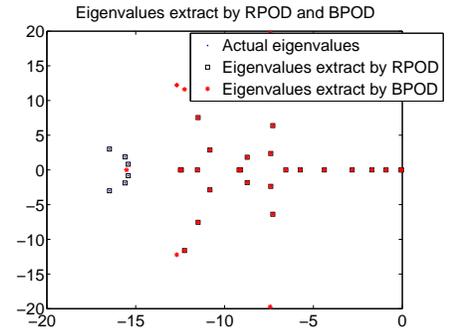}}
\subfigure[Comparison of state errors]{
\includegraphics[width=2.4in]{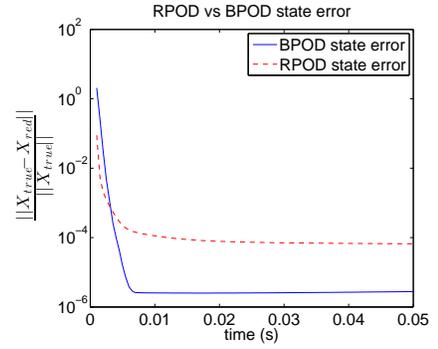}}
\caption{Comparison between RPOD and BPOD for 2D damped duffing oscillator}
\label{Fokker}
\end{figure}

We can see that the eigenvalues overlap the actual eigenvalues of the system, the state errors using BPOD  is around $0.0001\%$, and the state errors using RPOD are around $0.01\%$.

\subsection{Discussion}
We compare the computational requirements/accuracy of the ROMs resulting from the BPOD and RPOD for the Pollutant Transport equation (PT),
Channel Flow Problem (CF), and 2-D damped Duffing oscillator (DO) in Table \ref{table:results}. 


\begin{table}[htbp]
\caption{Comparison of SVD problem using BPOD $V.S.$ RPOD}
\centering
\begin{tabular}{|c|c|c|}
\hline
& size & average output error\\
\hline
PT& $(7500 \times 1500) : (1500 \times 900)$ & $0.055 \% : 0.6 \%$\\
\hline
CF & $(8000 \times 2000) : (2000 \times 400)$ & $0.13 \% : 0.16 \%$\\
\hline
DO & $(1500 \times 1500) : (294 \times 294) $& $0.007 \% : 0.017 \%$\\
\hline
\end{tabular}
\label{table:results}
\end{table}
We can see that RPOD solves a much smaller SVD problem than the BPOD, and although the errors incurred using RPOD are more than the BPOD, they are small enough not to make a major difference to the results. Thus, using the RPOD to generate a ROM is much more efficient while not sacrificing too much accuracy.

Moreover, sometimes, it may be impossible to solve the SVD problem resulting from BPOD. For example, in the linearized Channel flow problem, if  we use the full state measurements(882 measurements) and we take 20 snapshots for the adjoint simulation, there are 80 sources on the bounday and we take 1000 snapshots for the primal simulation, then we need to solve a $17640 \times 80000$ SVD problem for BPOD, which is not solvable in Matlab. For RPOD, we randomly choose 50 sources on the boundaries, and randomly choose 400 measurements. If we take 100 snapshots for the primal simulation, and 20 snapshots for the adjoint simulation, then it leads to a $8000 \times 5000$ SVD problem, which is a relatively small problem. We compare the first 70 extracted eigenvalues with the actual eigenvalues and  the output errors in Figure \ref{channel:new}. Thus, in problems where there are a large numbers of actuators/sensors, the savings can be very significant. In terms of an experiment, this observation may have added implications as it implies that we can reduce the scale of the instrumentation required to get the data required to form an ROM by orders of magnitude without losing much information that can be extracted from the resulting data, which can result in significant cost savings.

\begin{figure}[!htbp]
\centering
\subfigure[Eigenvalues extract by RPOD for channel flow problem]{
\includegraphics[width=2.4in]{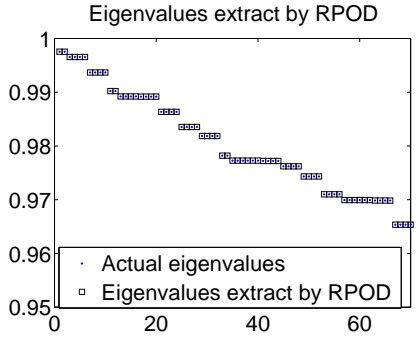}}
\subfigure[Output errors using RPOD]{
\includegraphics[width=2.4in]{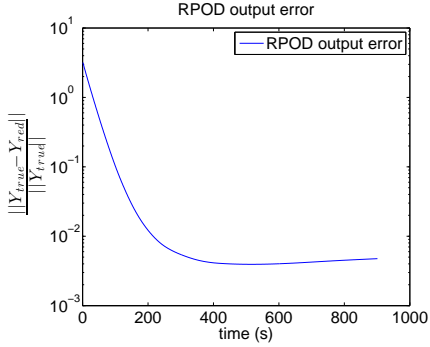}}
\caption{Simulation results using RPOD for channel flow problem}
\label{channel:new}
\end{figure}
\section{Conclusion}
In this paper, we have introduced a randomized POD (RPOD) procedure for the extraction of ROMs for large scale systems such as those governed by PDEs. The RPOD procedure extracts almost the same information from a randomly chosen sub-Hankel matrix extracted from the full order Hankel matrix as is obtained by the BPOD procedure from the full order Hankel matrix without sacrificing too much accuracy. This leads to an orders of magnitude reduction in the computation required for constructing ROMs for large scale systems with a large number of inputs/ outputs over the BPOD procedure. The computational results shown for a set of moderately high dimensional advection diffusion equations seem to reach the same conclusion. The next step in this process would require us to consider more realistic, high dimensional, and nonlinear PDEs arising in problems such as fluid flows and aeroelasticity. 
\appendix[Reconstructed eigenvalues' and eigenvectors' errors]\label{appendix A} Here, we establish bounds on the eigenfunction reconstruction errors using the cross correlation matrix $Y'X$. The eigenfunction reconstruction using the auto correlation matrix $X'X$ is a special case of this proof. 

We denote
\begin{eqnarray}
X=\left(\begin{array}{cc} V_S & V_D\end{array}\right) \left(\begin{array}{c} \alpha_S\\ \delta \alpha_D \end{array}\right) \nonumber \\
Y=\left(\begin{array}{cc} U_S & U_D\end{array}\right) \left(\begin{array}{c} \beta_S\\ \delta \beta_D \end{array}\right)
\end{eqnarray}
where $U_S$,$V_S$ are the active left and right eigenvectors corresponding to the same eigenvalues $\Lambda_S$ in the snapshots, and $U_D $, $V_D$ are rest of the left and right eigenvectors. As we have assumed before, $\| \delta \alpha_D \| \varpropto o(\epsilon)$, and $\| \delta \beta_D \| \varpropto o(\epsilon)$, where $\epsilon$ is sufficient small. First, we need to solve the SVD problem of $Y^TX$. 
\begin{eqnarray}
Y^TX= (\beta_S'U_S' + \delta \beta_D' U_D') (V_S \alpha_S+ V_D \delta \alpha_D) \nonumber \\
 = \beta_S'\alpha_S +  \underbrace{\delta \beta_D' \delta \alpha_D }_{\Delta_1} = \beta_S' \alpha_S + \Delta_1
\end{eqnarray}
where $\| \Delta_1 \| \varpropto o(\epsilon^2)$, and thus $\| Y^TX - \beta_S'\alpha_S \| \leq c_1 \epsilon^2$.
If $(U_p, \Sigma_p, V_p)$ are the left singular vectors, non-zero singular values and right singular vectors of $Y^TX$, i.e.
\begin{eqnarray}
Y^TX=U_p\Sigma_pV_p^T \nonumber \\
(\beta_S'U_S')(V_S \alpha_S) = \beta_S'\alpha_S = \hat{U}_p \hat{\Sigma}_p \hat{V}_p^T\label{svd_yx}
\end{eqnarray}
where $(\hat{U}_p, \hat{\Sigma}_p, \hat{V}_p)$ are the left singular vectors, non-zeros singular values, and right singular vectors of $\beta_S' \alpha_S$. From the eigenvalue perturbation theory, $\| V_p -\hat{V}_p \| \varpropto o(\epsilon^2)$, $\| U_p -\hat{U}_p \| \varpropto o(\epsilon^2)$,  $\| \Sigma_p -\hat{\Sigma}_p \| \varpropto o(\epsilon^2)$.

Thus,
\begin{eqnarray}
V_p \Sigma_p^{-1/2} = \hat{V}_p \hat{\Sigma}_p^{-1/2} +\Delta_2\nonumber \\
U_p \Sigma_p^{-1/2} = \hat{U}_p \hat{\Sigma}_p^{-1/2} + \Delta_3 \label{uv}
\end{eqnarray}
where $\| \Delta_2 \|, \| \Delta_3 \| \varpropto o(\epsilon^2)$.
The POD basis can be constructed as:
\begin{eqnarray}
T_r=XV_p\Sigma_p^{-1/2} \nonumber \\
T_l=\Sigma_p^{-1/2}U_p^TY^T 
\end{eqnarray}

We have:
\begin{eqnarray}
Y'AX= (\beta_S'U_S' + \delta \beta_D' U_D' )A ( V_S \alpha_S + V_D \delta \alpha_D)\nonumber \\
= \beta_S'\Lambda_S\alpha_S +\underbrace{\delta \beta_D' \Lambda_D \delta \alpha_D}_{\Delta_4} = \beta_S'\Lambda_S\alpha_S + \Delta_4 \label{yax}
\end{eqnarray}
where $\| \Delta_4 \|\varpropto o(\epsilon^2)$.
The reduced order system using this set of POD basis is:
\begin{eqnarray}
\tilde{A} = T_lAT_r= (\Sigma_p^{-1/2}U_p'){(Y'AX)}(V_p \Sigma_p^{-1/2}) \label{ta}
\end{eqnarray}

Substitute Equation (\ref{uv}) and Equation (\ref{yax}) into Equation (\ref{ta}), 
\begin{eqnarray}
\tilde{A} = T_lAT_r\nonumber \\
= (\hat{\Sigma}_p^{-1/2}\hat{U}_p' + \Delta_3)(\beta_S' \Lambda_S \alpha_S + \Delta_4)(\hat{V}_p \hat{\Sigma}_p^{-1/2} + \Delta_2) \nonumber \\
 =\underbrace{( \hat{\Sigma}_p^{-1/2}\hat{U}_p' \beta_S' )}_{P}\Lambda_S\underbrace{(\alpha_S\hat{V}_p \hat{\Sigma}_p^{-1/2})}_{\hat{P}}+ \Delta_5 = \hat{A} + \Delta_5 \label{AA}
\end{eqnarray}
where $\| \Delta_5 \| \varpropto o(\epsilon^2)$. We want to show $P\hat{P} = I$
\begin{eqnarray}
P\hat{P}= \hat{\Sigma}_p^{-1/2}\underbrace{\hat{U}_p' \hat{U}_p}_{I}\hat{\Sigma}_p\underbrace{\hat{V}_p'\hat{V}_p}_{I} \hat{\Sigma}_p^{-1/2}=\hat{\Sigma}_p^{-1/2}\hat{\Sigma}_p \hat{\Sigma}_p^{-1/2}=I
\end{eqnarray}

Since $P$ and $\hat{P}$ are square matrices, thus, $\hat{P} = P^{-1}$, $\hat{A} = P \Lambda_S P^{-1}$.
From Equation (\ref{AA}), 
\begin{eqnarray}
\tilde{A} = \tilde{P} \Lambda_{ij} \tilde{P}^{-1}= \hat{A} + \Delta_5
\end{eqnarray}
Using the eigenvalue perturbation theory, $\tilde{P} = P + \Delta_6$, where $\| \Delta_6 \| \ \varpropto o(\epsilon^2)$, $\| \Lambda_{ij} - \Lambda_S \| \varpropto o(\epsilon^2)$. where $\Lambda_S $ are the eigenvalues of the system matrix $A$.
 Now, we want to bound the errors between the right and left eigenvectors corresponding to the same eigenvalues.
\begin{eqnarray}
\Psi_{ij} = T_r\tilde{P}=XV_p \Sigma_p^{-1/2} (P + \Delta_6)\\
=(V_S \alpha_S+ V_D \delta \alpha_D)(\hat{V}_p \hat{\Sigma}_p^{-1/2} +\Delta_2)(P+\Delta_6).\nonumber\\
=(V_S \alpha_S + V_D \delta \alpha_D )( \hat{V}_p \hat{\Sigma}_p^{-1}\hat{U}_p' \beta_S'  +\Delta_7) \\
=V_S \underbrace{\alpha_S \hat{V}_p \hat{\Sigma}_p^{-1}\hat{U}_p' \beta_S'}_{P^{-1}P}+V_D \delta \alpha_D  \hat{V}_p \hat{\Sigma}_p^{-1}\hat{U}_p' \beta_S' + \Delta_8 \nonumber\\
=V_S + V_D \delta \alpha_D  \hat{V}_p \hat{\Sigma}_p^{-1}\hat{U}_p' \beta_S' +\Delta_8 \nonumber 
\end{eqnarray}
Here, $\| \Delta_7 \|, \| \Delta_8 \|\ \varpropto o(\epsilon^2)$. Since $\|V_D \delta \alpha_D  \hat{V}_p \hat{\Sigma}_p^{-1}\hat{U}_p' \beta_S' \| \varpropto o(\epsilon) $, then $\| \Psi_{ij}- V_S \| \varpropto o(\epsilon)$. Similarly, if we denote $\Phi_{ij}' = \tilde{P}^{-1} T_l$, then $\|\Phi_{ij}- U_S\| \varpropto o(\epsilon)$.

\bibliographystyle{IEEEtran}
\bibliography{IPOD_refs}

\begin{thebibliography}{10}
\providecommand{\url}[1]{#1}
\csname url@samestyle\endcsname
\providecommand{\newblock}{\relax}
\providecommand{\bibinfo}[2]{#2}
\providecommand{\BIBentrySTDinterwordspacing}{\spaceskip=0pt\relax}
\providecommand{\BIBentryALTinterwordstretchfactor}{4}
\providecommand{\BIBentryALTinterwordspacing}{\spaceskip=\fontdimen2\font plus
\BIBentryALTinterwordstretchfactor\fontdimen3\font minus
  \fontdimen4\font\relax}
\providecommand{\BIBforeignlanguage}[2]{{%
\expandafter\ifx\csname l@#1\endcsname\relax
\typeout{** WARNING: IEEEtran.bst: No hyphenation pattern has been}%
\typeout{** loaded for the language `#1'. Using the pattern for}%
\typeout{** the default language instead.}%
\else
\language=\csname l@#1\endcsname
\fi
#2}}
\providecommand{\BIBdecl}{\relax}
\BIBdecl

\bibitem{kung}
S.~Kung, ``A new identification and model reduction algorithm via singular
  value decomposition,'' \emph{12th Asilomar Conference on Circuits, Systems
  and Computers}, pp. 705--714, Nov. 1978.

\bibitem{juang}
J.-N. Juang, \emph{Applied System Identification}.\hskip 1em plus 0.5em minus
  0.4em\relax Englewood Cliffs, NJ: Prentice Hall, 1994.

\bibitem{pod2}
G.~Berkooz \emph{et~al.}, ``The proper orthogonal decomposition in the analysis
  of turbulent flows,'' \emph{Ann. Rev.. Fl. Mech.}, vol.~25, pp. 539--575,
  1993.

\bibitem{pod3}
K.~C. Hall \emph{et~al.}, ``Proper orthogonal decomposition technique for
  transonic unsteady aerodynamic flows,'' \emph{AIAA Journal}, vol.~38, pp.
  1853--1862, 2000.

\bibitem{pod4}
L.~Sirovich, ``Turbulence and dynamics of coherent structures. part 1: Coherent
  structures,'' \emph{Quarterly of Applied Mathematics}, vol.~45, pp. 561--571,
  1987.

\bibitem{moore}
B.~C. Moore, ``Principal component analysis in linear systems: Controllability,
  observability and model reduction,'' \emph{IEEE Transactions on Automatic
  Control}, vol.~26, pp. 17--32, 1981.

\bibitem{willcox}
K.~Willcox and J.~Peraire, ``Balanced model reduction via the proper orthogonal
  decomposition,'' \emph{AIAA Journal}, vol.~40, pp. 2323--2330, 2002.

\bibitem{rowley1}
C.~W. Rowley, ``Model reduction of fluids using balanced proper orthogonal
  decomposition,'' \emph{International Journal fo Bifurcation and Chaos},
  vol.~15, pp. 997--1013, 2005.

\bibitem{juang1985}
J.~Jer-Nan and R.~Pappa, ``An eigensystem realization algorithm for model
  parameter identification and model reduction,'' \emph{Journal of Guidance,
  Control, and Dynamics}, vol. 8, No.5, pp. 620--627, 1985.

\bibitem{rowley4}
Z.~Ma \emph{et~al.}, ``Reduced order models for control of fluids using the
  eigensystem realization algorithm,'' \emph{Theoret. Comput. Fluid Dyn.},
  vol.~36, pp. 233--247, 2006.

\bibitem{schmid}
P.~J. Schmid, ``Dynamic mode decomposition of numerical and experimental
  data,'' \emph{Journal of Fluid Mechanics}, vol. 656, pp. 5--28, 2010.

\bibitem{rowley2}
C.~W. Rowley \emph{et~al.}, ``Spectral analysis of nonlinear analysis,''
  \emph{Journal of Fluid Mechanics}, vol. 641, pp. 115--127, 2009.

\bibitem{antoulas}
A.~C. Antoulas, \emph{Approximation of Large Scale Dynamical Systems}.\hskip
  1em plus 0.5em minus 0.4em\relax Philadelphia: SIAM, 2005.

\bibitem{RC1}
M.~Vidyasagar, ``Randomized algorithms for robust controller synthesis using
  statistical learning theory,'' \emph{Automatica}, vol.~37, pp. 1515--1528,
  2001.

\bibitem{RC2}
R.~Tempo \emph{et~al.}, ``Probabilistic robustness analysis: Explicit bounds
  for minimum number of samples,'' \emph{Systems and Control Letters}, vol.~30,
  pp. 237--242, 1997.

\bibitem{RC3}
L.~R. Ray and R.~F. Stengel, ``A monte carlo approach to the analysis of
  control system robustness,'' \emph{Automatica}, vol.~29, pp. 229--236, 1993.

\bibitem{RC4}
B.~T. Polyak and R.~Tempo, ``Probabilistic robust design with linear quadratic
  regulators,'' \emph{Systems and Control Letters}, vol.~43, pp. 343--353,
  2001.

\bibitem{RC5}
G.~Calafiore and M.~Campi, ``The scenario approach to robust control design,''
  \emph{IEEE Transactions on Automatic Control}, vol.~51, pp. 742--753, 2006.

\bibitem{RC6}
M.~C. Campi, S.~Garatti, and M.~Prandini, ``The scenario approach for systems
  and control design,'' \emph{Ann. Rev. Control}, vol.~33, pp. 149--157, 2009.

\bibitem{ACC}
D.Yu and S.Chakravorty, ``A randomized iterative proper orthogonal
  decomposition technique with application to filtering of pdes,''
  \emph{Proceedings of American Control Conference}, pp. 4363--4368, 2012.

\bibitem{JAS}
------, ``An iterative proper orthogonal decomposition (i-pod) technique with
  application to the filtering of partial differential equations,''
  \emph{Journal of Astronautical Sciences}, vol. Special issue on J. N. Juang's
  60th birthday, to appear, 2013.

\bibitem{eigp}
T.~Kato, \emph{Perturbation Theory for Linear Operators}.\hskip 1em plus 0.5em
  minus 0.4em\relax New York: Springer-Verlag, 1995.

\end{thebibliography}

\end{document}